\documentclass[10pt,reqno]{amsart}
\usepackage{amsaddr}
\usepackage[a4paper, total={6.5in, 8.4in}]{geometry}

\usepackage[symbol]{footmisc}

\usepackage{amsmath}
\usepackage{amssymb}
\usepackage{amsthm}
\usepackage[latin1]{inputenc}
\usepackage{eurosym}
\usepackage{graphicx}
\usepackage{epsfig}
\usepackage{hyperref}
\usepackage{dsfont}
\usepackage[space]{cite}
\usepackage{subcaption}
\usepackage{caption}
\usepackage{placeins}
\usepackage{mathtools}
\usepackage{xcolor}
\usepackage[ruled,vlined]{algorithm2e}
\usepackage[title]{appendix}

\usepackage[displaymath,mathlines]{lineno}
\modulolinenumbers[1]

\allowdisplaybreaks

\usepackage{hyperref}
\usepackage{ifthen}
\usepackage{comment}

\makeindex
\usepackage{color}
\usepackage{float}
\usepackage{commath}

\newcommand{\argmin}{\operatorname{argmin}}

\def\leq{\leqslant}
\def\geq{\geqslant}

\numberwithin{equation}{section}

\newtheoremstyle{thmlemcorr}{10pt}{10pt}{\itshape}{}{\bfseries}{.}{10pt}{{\thmname{#1}\thmnumber{
			#2}\thmnote{ (#3)}}}
\newtheoremstyle{thmlemcorr*}{10pt}{10pt}{\itshape}{}{\bfseries}{.}\newline{{\thmname{#1}\thmnumber{
			#2}\thmnote{ (#3)}}}
\newtheoremstyle{defi}{10pt}{10pt}{\itshape}{}{\bfseries}{.}{10pt}{{\thmname{#1}\thmnumber{
			#2}\thmnote{ (#3)}}}
\newtheoremstyle{remexample}{10pt}{10pt}{}{}{\bfseries}{.}{10pt}{{\thmname{#1}\thmnumber{
			#2}\thmnote{ (#3)}}}
\newtheoremstyle{ass}{10pt}{10pt}{}{}{\bfseries}{.}{10pt}{{\thmname{#1}\thmnumber{
			A#2}\thmnote{ (#3)}}}

\theoremstyle{thmlemcorr}
\newtheorem{theorem}{Theorem}
\numberwithin{theorem}{section}
\newtheorem{lemma}[theorem]{Lemma}
\newtheorem{corollary}[theorem]{Corollary}

\theoremstyle{thmlemcorr*}
\newtheorem{theorem*}{Theorem}
\newtheorem{lemma*}[theorem]{Lemma}
\newtheorem{corollary*}[theorem]{Corollary}
\newtheorem{proposition*}[theorem]{Proposition}
\newtheorem{problem*}[theorem]{Problem}
\newtheorem{conjecture*}[theorem]{Conjecture}

\theoremstyle{defi}

\newtheorem{hyp}{Assumption}

\theoremstyle{remexample}
\newtheorem{remark}[theorem]{Remark}

\newtheorem{lem}[theorem]{Lemma}
\newtheorem{pro}[theorem]{Proposition}

\theoremstyle{ass}

\newtheorem*{notations*}{Notations}
\newtheorem*{acknowledgement*}{Acknowledgement}
\newtheorem*{relatedwork*}{Other related work}
\allowdisplaybreaks

\title[A MINI-BATCH METHOD For GAUSSIAN
PROCESS REGRESSION]{A MINI-BATCH METHOD For SOLVING NONLINEAR PDES WITH GAUSSIAN PROCESSES}
\author{Xianjin Yang$^{1,*}$, Houman Owhadi$^1$}

\address[X. Yang]{
	$^1$Computing \& Mathematical Sciences, California Institute of Technology, Pasadena, CA 91125.}
 
\email{yxjmath@caltech.edu}
\email{owhadi@caltech.edu}

\thanks{$^*$Corresponding author.}

\begin{document}
\maketitle

\begin{abstract}
Gaussian processes (GPs) based methods for solving partial differential equations (PDEs)  demonstrate great promise by bridging the gap between the theoretical rigor of traditional numerical algorithms and the flexible design of machine learning solvers. The main bottleneck of GP methods lies in the inversion of a covariance matrix, whose cost grows cubically concerning the size of samples. Drawing inspiration from neural networks, we propose a mini-batch algorithm combined with GPs to solve nonlinear PDEs. {\color{black}A naive deployment of a stochastic gradient descent method for solving PDEs with GPs is challenging, as the objective function in the requisite minimization problem cannot be depicted as the expectation of a finite-dimensional random function. To address this issue, we employ a mini-batch method to the corresponding infinite-dimensional minimization problem over  function spaces.}
The algorithm takes a mini-batch of samples at each step to update the GP model. Thus, the computational cost is allotted to each iteration. Using stability analysis and convexity arguments, we show that the mini-batch method steadily reduces a natural measure of errors towards zero at the rate of $O(1/K+1/M)$, where $K$ is the number of iterations and $M$ is the batch size.
\end{abstract}

\section{Introduction}
	This work develops a mini-batch method based on stochastic proximal algorithms \cite{davis2019stochastic, asi2019importance, asi2019stochastic} to solve nonlinear partial differential equations (PDEs) with Gaussian processes (GPs). PDEs have been widely used in science, economics, and biology \cite{quarteroni2008numerical, thomas2013numerical}. Since PDEs generally lack explicit solutions, numerical approximations are essential in applications. Standard numerical methods, for instance, finite difference \cite{thomas2013numerical} and finite element methods \cite{hughes2012finite},  for solving nonlinear PDEs are prone to the curse of dimensions. Recently, to keep pace with scaling problem sizes, machine learning methods,  including neural networks \cite{raissi2019physics, lu2021learning, karniadakis2021physics} and probabilistic algorithms \cite{raissi2018numerical, kramer2021probabilistic, chen2021solving}, arise a lot of attention. Especially, GPs present a promising approach that integrates the rigorous principles of traditional numerical algorithms with the adaptable framework of machine learning solvers \cite{raissi2018numerical, chen2021solving, wang2021bayesian, batlle2023error}. Our methods are based on the work of \cite{chen2021solving}, which proposes a GP framework for solving general nonlinear PDEs. A numerical approximation of the solution to a nonlinear PDE is viewed as the maximum \textit{a posteriori} (MAP) estimator of a GP conditioned on solving the PDE at a finite set of sample points. Equivalently, \cite{chen2021solving} proposes to approximate the unique strong solution of a nonlinear PDE in a reproducing kernel Hilbert space (RKHS) $\mathcal{U}$ by a minimizer of the following optimal recovery problem
\begin{align}
	\label{intro_nlprob}
	\min_{u\in \mathcal{U}}\frac{\lambda}{2}\|u\|_{\mathcal{U}}^2 + \frac{1}{N}\sum_{i=1}^N|f_i([\boldsymbol{\phi}_i, u])-{y}_i|^2,
\end{align}
where  $\lambda$ is a small regularization parameter, $N$ is the number of samples  in the domain, $f_i[\boldsymbol{\phi}_i, u]-y_i$ encodes the PDE equation satisfied at the $i^{\textit{th}}$ sample, and $\boldsymbol{\phi}_i$ represents the linear operators in the $i^{\textit{th}}$ equation. Let $\boldsymbol{\phi}=(\boldsymbol{\phi}_1,\dots, \boldsymbol{\phi}_N)$, let $\kappa$ be the kernel associated with  $\mathcal{U}$, and let $\Omega$ be the domain of the PDE. Denote by $\kappa(\boldsymbol{x}, \boldsymbol{\phi})$ the action of $\boldsymbol{\phi}$ on $\kappa(\boldsymbol{x},\cdot)$ for fixed $\boldsymbol{x}\in \Omega$ and by $\kappa(\boldsymbol{\phi}, \boldsymbol{\phi})$ the action of $\boldsymbol{\phi}$ on $\kappa(\cdot, \boldsymbol{\phi})$.    By the representer theorem (see \cite[Sec. 17.8]{owhadi2019operator}), a solution $u^*$ of \eqref{intro_nlprob} satisfies
	$u^*(\boldsymbol{x}) = \kappa(\boldsymbol{x}, \boldsymbol{\phi})\kappa(\boldsymbol{\phi, \boldsymbol{\phi}})^{-1}\boldsymbol{z}^*, \forall \boldsymbol{x}\in \Omega$, 
where $\boldsymbol{z}^*$ solves
\begin{align}
	\label{intro_zprob}
	\min_{\boldsymbol{z}} \frac{\lambda}{2}\boldsymbol{z}^T\kappa(\boldsymbol{\phi}, \boldsymbol{\phi})^{-1}\boldsymbol{z} + \frac{1}{2N}\sum_{i=1}^N|f_i(\boldsymbol{z}_i)-y_i|^2. 
\end{align}
Then, \cite{chen2021solving} proposes to solve \eqref{intro_zprob} by the Gauss--Newton method. {\color{black}We refer readers to \cite{chen2021solving, meng2023sparse} for illustrative examples on how to transform a PDE solving task into the minimization problem described in \eqref{intro_nlprob}.}
The computational cost of the above GP method grows cubically with respect to the number of samples due to the inversion of the covariance matrix $\kappa(\boldsymbol{\phi}, \boldsymbol{\phi})$. To improve the performance, based on \cite{chen2021solving}, \cite{mou2022numerical} proposes to approximate the solution of a PDE in a function space generated by random Fourier features \cite{rahimi2007random, yu2016orthogonal}, where the corresponding covariance matrix is a low-rank approximation to $\kappa(\boldsymbol{\phi}, \boldsymbol{\phi})$. In \cite{meng2023sparse}, another low-rank approximation method using sparse GPs is presented, employing a small number of inducing points to summarize observation information. Additionally, \cite{chen2023sparse} introduces an algorithm that efficiently computes the sparse Cholesky factorization of the inverted Gram matrix. Despite their computational cost-saving accomplishments, low-rank approximation methods lack the scalability of neural network-based algorithms. {\color{black}The primary cause for this difference lies in the scaling of parameters in the GP/Kernel method with respect to the number of collocation points, which in turn escalates with the increase in dimensions. Conversely, in neural network models, the parameters are designed to be independent of the number of data points.}

{\color{black}
Compared to methods based on neural networks \cite{raissi2019physics, mao2020physics, karniadakis2021physics} , Gaussian Process (GP)/Kernel-based approaches generally offer solid theoretical foundations. Additionally, from a probabilistic viewpoint, GP regressions provide uncertainty quantification. This quantification allows us to establish the confidence intervals of the results and aids in deciding the placement of collocation points. The exploration of how to effectively utilize uncertainty quantification will be a subject for future research.  Nevertheless, the GP method typically incurs a computational cost of $O(N^3)$, where 
$N$ represents the number of collocation points. On the other hand, besides their capacity to model a wide range of functions, neural network methods also excel in scalability, largely due to the use of stochastic gradient descent (SGD) algorithms with mini-batches. This scalability is facilitated by the independence of their data points from the parameters. However, due to the global correlation of the unknowns in   \eqref{intro_zprob}, the direct application of SGD is impeded by the need to solve the inversion of the covariance matrix in equation \eqref{intro_zprob}. Consequently, devising an effective SGD approach is vital for the GP methodology to reach a level of scalability on par with neural networks.} This paper attempts to handle this issue by proposing a mini-batch method inspired by stochastic proximal methods \cite{davis2019stochastic, asi2019importance, asi2019stochastic}. Rather than pursuing low-rank approximations, we stochastically solve equation  \eqref{intro_nlprob}. Each iteration selects a small portion of samples to update the GP model. Hence, we only need to consider the inversion of the covariance matrix of samples in a mini-batch each time. More precisely, we introduce a slack variable $\boldsymbol{z}$ and reformulate \eqref{intro_nlprob} as
\begin{align}
	\label{intro_mbnlprob}
	\min_{u\in \mathcal{U}, \boldsymbol{z}\in \mathcal{X}}\frac{\lambda}{2}\|u\|_{\mathcal{U}}^2 + \frac{\beta}{2N}\sum_{i=1}^N|[\boldsymbol{\phi}_i, u]-\boldsymbol{z}_i|^2 + \frac{1}{2N}\sum_{i=1}^N|f_i(\boldsymbol{z}_i)-{y}_i|^2,
\end{align}
where $\mathcal{X}$ is a convex compact set, $\beta>0$, and $\boldsymbol{z}=(\boldsymbol{z}_1,\dots,\boldsymbol{z}_N)$. We note that as $\beta\rightarrow \infty$, the solutions of \eqref{intro_mbnlprob} converge to those of \eqref{intro_nlprob}. The latent variable $\boldsymbol{z}$ encodes observations of $u$. Next, we view \eqref{intro_mbnlprob} as a stochastic optimization problem
\begin{align}
	\label{intro_stombnlprob}
	\min_{u\in \mathcal{U}, \boldsymbol{z}\in \mathcal{X}}E_{\xi}\bigg[\frac{\lambda}{2}\|u\|_{\mathcal{U}}^2 + \frac{\beta}{2}|[\boldsymbol{\phi}_\xi, u]-\boldsymbol{z}_\xi|^2 + \frac{1}{2}|f_\xi(\boldsymbol{z}_\xi)-{y}_\xi|^2\bigg],
\end{align}
where $\xi$ is a uniform random sample taken from the set of the whole samples. Then,  we iteratively solve \eqref{intro_stombnlprob} by computing the proximal map
\begin{align}
	\label{intro_stopm}
	(u^{k+1}, \boldsymbol{z}^{k+1}) =& \argmin\limits_{u\in \mathcal{U}, \boldsymbol{z}\in \mathcal{X}}\frac{\lambda}{2}\|u\|_{\mathcal{U}}^2 + \frac{\beta}{2M}\sum_{i\in \mathcal{I}_k}|[\boldsymbol{\phi}_i, u]-\boldsymbol{z}_i|^2\nonumber \\
	& + \frac{1}{2M}\sum_{i\in \mathcal{I}_k}|f_i(\boldsymbol{z}_i)-{y}_i|^2 + \frac{\gamma_k}{2}|\boldsymbol{z} - \boldsymbol{z}^k|^2,
\end{align}
where $\mathcal{I}_k$ contains a mini-batch of samples, $M$ is the cardinality of $\mathcal{I}_k$, and $\gamma_k$ is the step size at the $k^{\textit{th}}$ step. Denote by $\boldsymbol{v}_{\mathcal{I}_k}$ the vector containning elements of a vector $\boldsymbol{v}$ indexed by $\mathcal{I}_k$. Then, to handle the inifnite dimensional minimization problem in \eqref{intro_stopm}, we derive a representer formula for \eqref{intro_stopm} in Section \ref{secGPNM} such that 
\begin{align}
	\label{intro_nluzprli}
	u^{k+1}(\boldsymbol{x})=\kappa(\boldsymbol{x}, \boldsymbol{\phi}_{\mathcal{I}_k})\bigg(\kappa(\boldsymbol{\phi}_{\mathcal{I}_k}, \boldsymbol{\phi}_{\mathcal{I}_k}) + \frac{\lambda M}{\beta}{I} \bigg)^{-1}\boldsymbol{z}^{k+1}_{\mathcal{I}_k}, \boldsymbol{x} \in \Omega, 
\end{align}
where ${I}$ is the identity matrix,  $\boldsymbol{z}^{k+1}=(\boldsymbol{z}^{k}_{\mathcal{N}\backslash\mathcal{I}_k}, \boldsymbol{z}^{k+1}_{\mathcal{I}_k})$, and  $\boldsymbol{z}^{k+1}$ solves
\begin{align}
	\label{intro_nlzdexpl}
	\min_{\boldsymbol{z} \in \mathcal{X}}\frac{\lambda}{2}\boldsymbol{z}^T_{\mathcal{I}_k}\bigg(\kappa(\boldsymbol{\phi}_{\mathcal{I}_k}, \boldsymbol{\phi}_{\mathcal{I}_k})
	+ \frac{\lambda M}{\beta}I\bigg)^{-1}\boldsymbol{z}_{\mathcal{I}_k}
	 + \frac{1}{2M}\sum_{i\in \mathcal{I}_k}|f_i(z_i) - y_i|^2 +  \frac{\gamma_{k}}{2}|\boldsymbol{z}_{\mathcal{I}_k}-{\boldsymbol{z}}_{\mathcal{I}_k}^{k}|^2.  
\end{align}
Thus, the computational cost is reduced to $O(M^3)$ in each iteration if we use the standard Cholesky decomposition to compute the inverse of the covariance matrix for solving \eqref{intro_nluzprli} and \eqref{intro_nlzdexpl}. Under assumptions of boundedness of  $\boldsymbol{\phi}$ and the weakly convexity of the functions $\{f_i\}_{i=1}^N$, our convergence analysis shows that the error diminishes at the rate of {\color{black}$O(1/K + 1/M)$, where $K$ is the number of iterations and $M$ is the batch size. (For clarification, it's important to note that the computational complexity of addressing the inner minimization issue, as presented in \eqref{intro_stopm}, stands at $O(M^3)$. Meanwhile, the complexity associated with the outer iteration, which is relevant for the convergence of $(u^{k+1}, \boldsymbol{z}^{k+1})$, is represented as $O(1/K+1/M)$.) } 

{\color{black}To illustrate our contributions, we present a comparison of our mini-batch approach with highly related existing mini-batch methods in the literature. Notably, the studies by  \cite{chen2020stochastic, chen2022gaussian}  introduce SGD techniques for hyperparameter learning in GP regressions. A key distinction of our method lies in its treatment of the minimization over the latent vector $\boldsymbol{z}$, as outlined in \eqref{intro_zprob}. In contrast, \cite{chen2020stochastic, chen2022gaussian} focus on optimizing kernel hyperparameters for given data $\boldsymbol{z}$, without considering this minimization.
In alignment with our theoretical findings, the studies \cite{chen2020stochastic, chen2022gaussian} similarly identify errors comprising two main elements: an optimization error term and a statistical error term. These are determined by the number of optimization  iterations and the mini-batch size. Consequently, with a fixed batch size, it becomes  challenging to achieve zero error as the number of optimization steps grows. A key reason that we cannot eliminate the statistical error is the exclusion of a proximal step for $u$ as seen in \eqref{intro_stopm}. If we were to integrate the proximal step for $u$, it would require maintaining all mini-batch samples and making the $k^{\textit{th}}$ iteration's minimizer $u^k$ depending on the  expressions of all previous minimizers $u^1,\dots,u^{k-1}$. This integration would lead to significantly higher memory demands (see also Remark \ref{rmk:stcerr}).   Exploring ways to reduce statistical error while balancing memory efficiency remains an area for future investigation.

Furthermore, \cite{tran2020stochastic} develops a stochastic Gauss-Newton algorithm for non-convex compositional optimization. However, their approach is not applicable in our context due to the quadratic nature of \eqref{intro_zprob}, which is incompatible with their algorithm's requirement for an expectation-based formulation in finite-dimensional settings.  This is evident as the term $\boldsymbol{z}^T\kappa(\boldsymbol{\phi}, \boldsymbol{\phi})^{-1}\boldsymbol{z}$ in Equation (1.2) cannot be framed as an expectation over $\boldsymbol{z}$. Our method employs a mini-batch approach solving the infinite-dimensional minimization problem in \eqref{intro_nlprob}.

Another subtle difference is our methodological approach compared to that of \cite{tran2020stochastic},. While they initiate their process by linearizing the objective function and then applying mini-batch updates to this linearized problem, we start by formulating a nonlinear proximal minimization step in each iteration. For solving this, we primarily use the Gauss-Newton method. It's important to note, however, that the Gauss-Newton method in our framework can be substituted with any other efficient optimization algorithm suitable for the proximal iteration step.
}

We organize the paper as follows. Section \ref{secGPNM}  summarizes the GP method in \cite{chen2021solving} and proposes a mini-batch method to solve PDEs in a general setting. In Section \ref{secConverge}, we extend the arguments in the framework of \cite{davis2019stochastic, bousquet2002stability, deng2021minibatch} and present a convergence analysis of the min-batch algorithm based on stability analysis and convexity arguments. 
Numerical experiments appear in Section \ref{secNumResults}. 
Further discussions and future work appear in Section \ref{secConclu}. 

\begin{relatedwork*}
	Though GP regressions have achieved great success in various applications, the main drawback of it is the $O(N^3)$ computational time and $O(N^2)$ memory consumption for $N$ training points. Several techniques have been introduced to reduce the computational cost in recent years. One branch of these methods is to introduce sparsity into kernels of GPs based on variational inference procedures \cite{quinonero2005unifying, damianou2016variational, Hensman_2013, lazaro2009inter, wilson2015kernel, liu2020gaussian} or random Fourier features \cite{hensman2017variational, yu2016orthogonal, rahimi2007random, lazaro2010sparse}. Another trend in tackling the scalability of the GP regression is to utilize increasing computational power and GPU acceleration to solve exact GPs \cite{nguyen2019exact, gardner2018gpytorch, wang2019exact}. Alternatively, a sparse approximation for the inverse Cholesky factor of a dense covariance matrix is considered in 
	\cite{schafer2021sparse}. Recently, stochastic gradient descent methods (SGDs) have also been taken into consideration for GPs. The work \cite{chen2020stochastic} proposes a SGD method to perform hyperparameter learning in GP regressions. The critical difference between our method and the work in \cite{chen2020stochastic} is that the problem in \cite{chen2020stochastic}  does not take into consideration the minimization over the latent vector $\boldsymbol{z}$ in \eqref{intro_zprob}. Instead, the authors in \cite{chen2020stochastic} optimize over the hyperparameters of the kernel for known data  $\boldsymbol{z}$. 
	
	SGDs are widely used in stochastic optimization problems \cite{bottou2007tradeoffs, nemirovski2009robust, shalev2007pegasos, zinkevich2003online}. However, since the unknowns in \eqref{intro_zprob} are globally correlated, SGDs cannot be applied without solving the inversion of the covariance matrix in \eqref{intro_zprob}. Moreover, SGD methods are sensitive to the choice of stepsize and may diverge if objective functions do not satisfy the convergence criteria \cite{asi2019importance, asi2019stochastic, li2017hyperband}. Recently, stochastic proximal point and model-based methods \cite{asi2019stochastic, bertsekas2011incremental, davis2019stochastic, duchi2018stochastic, kulis2010implicit} are proposed to serve as robust alternatives to standard stochastic gradient methods. Several works \cite{duchi2018stochastic, davis2019stochastic} show that the stochastic proximal point and model-based methods exhibit the robustness to stepsize choice and the convergence on a broader range of difficult problems. 
\end{relatedwork*}

\begin{notations*}
Given an index set $\mathcal{I}$, we denote by $|\mathcal{I}|$  the cardinality of $\mathcal{I}$.  For a real-valued vector $\boldsymbol{v}$, we represent by $|\boldsymbol{v}|$ the Euclidean norm of $\boldsymbol{v}$ and by $\boldsymbol{v}^T$ its transpose. Let $\langle \boldsymbol{u}, \boldsymbol{v}\rangle$ or $\boldsymbol{u}^T\boldsymbol{v}$ be the inner product of vectors $\boldsymbol{u}$ and $\boldsymbol{v}$.  For a normed vector space $V$, let $\|\cdot\|_V$ be the norm of $V$. 
	Let $\mathcal{U}$ be a Banach space associated with a quadratic norm $\|\cdot\|_{\mathcal{U}}$ and let $\mathcal{U}^*$ be the dual of $\mathcal{U}$. Denote by $[\cdot, \cdot]$ the duality pairing between $\mathcal{U}^*$ and $\mathcal{U}$, and by $\langle \cdot, \cdot\rangle_{\mathcal{U}}$ the inner produce of $\mathcal{U}$.  Suppose  that there exists a linear, bijective, symmetric ($[\mathcal{K}_{\mathcal{U}}\phi, \psi]=[\mathcal{K}_{\mathcal{U}}\psi, \phi]$), and positive ($[\mathcal{K}_{\mathcal{U}}\phi, \phi]>0$ for $\phi\not=0$) covariance operator $\mathcal{K}_{\mathcal{U}}: \mathcal{U}^*\mapsto \mathcal{U}$, such that 
		$\|u\|_{\mathcal{U}}^2 = [\mathcal{K}^{-1}u, u], \forall u\in \mathcal{U}$. 
	Let $\{\phi_i\}_{i=1}^P$ be a set of $P\in\mathbb{N}$ elements in $\mathcal{U}^*$ and let  $\boldsymbol{\phi}:=(\phi_1, \dots, \phi_P)$ be in the product space ${(\mathcal{U}^*)}^{\bigotimes P}$. Then, for $u\in \mathcal{U}$, we write $[\boldsymbol{\phi}, u]:=([{\phi}_1, u], \dots, [{\phi}_P, u])$.
	Let $\|\boldsymbol{\phi}\|_{\mathcal{U}^*}^2=\sum_{i=1}^P\|{\phi}_i\|_{\mathcal{U}^*}^2$. 
	Furthermore, for $\boldsymbol{u}:=(u_1,\dots,u_S)\in \mathcal{U}^{\bigotimes S}$, $S\in \mathbb{N}$,  let $[\boldsymbol{\phi}, \boldsymbol{u}]\in\mathbb{R}^{P\times S}$ be the matrix with entries $[\phi_i, u_j]$. 
	Finally, we represent by $C$ a positive real number whose value  may change line by line. 
\end{notations*}

\section{GP Regressions with Nonlinear Measurements}
\label{secGPNM}
For ease of presentation, in this section, we introduce an iterative mini-batch method for solving nonlinear PDEs in the general setting of GP regressions with nonlinear measurements. We first summarize the GP method proposed by \cite{chen2021solving} in Subsection \ref{subsecRevisit}. Then, in Subsection \ref{subsecGPnl}, we detail the min-batch method and derive a representer theorem, from which we see that the mini-batch method allocates the burden of computations to each iteration. 
\subsection{A Revisit of The GP Method}
\label{subsecRevisit}
Here, we summarize the general framework of the GP method \cite{chen2021solving} for solving nonlinear PDEs. 
Let $\Omega\subset\mathbb{R}^d$ be an open bounded domain with the boundary $\partial \Omega$ for $d\geq 1$. Let $L_1,\dots, L_{D_b}, L_{D_b+1}, \dots, L_D$ be bounded linear operators for $1\leq D_b\leq D$, where $D_b, D\in \mathbb{N}$. We solve
\begin{align}
	\label{pdegrf}
	\begin{cases}
		{P}(L_{D_b+1}(u^*)(\boldsymbol{x}),\dots, L_D(u^*)(\boldsymbol{x})) = f(\boldsymbol{x}), \forall \boldsymbol{x}\in \Omega,\\
		{B}(L_1(u^*)(\boldsymbol{x}),\dots, L_{D_b}(u^*)(\boldsymbol{x})) = g(\boldsymbol{x}), \forall \boldsymbol{x}\in \partial \Omega,
	\end{cases}
\end{align}
where $u^*$ is the unknown function, $P$ and $B$ represent nonlinear operators, and $f$, $g$ are given data. Throughout this paper, we suppose that \eqref{pdegrf} admits a unique strong solution in a RKHS $\mathcal{U}$ associated with the covariance operator $\mathcal{K}$. 

To approximate $u^*$ in $\mathcal{U}$, we first take $N$ sample points $\{{\boldsymbol{x}}_j\}_{j=1}^N$ in $\overline{\Omega}$ such that $\{{\boldsymbol{x}}_j\}_{j=1}^{N_\Omega}\subset\Omega$ and $\{{\boldsymbol{x}}_j\}_{j=N_{\Omega}+1}^N\subset\partial\Omega$ for $1\leq N_\Omega\leq N$. Then, we define 
\begin{align*} 
	\phi^{(i)}_j:=\delta_{{\boldsymbol{x}}_j}\circ L_i, \text{ and } \begin{cases}
		N_{\Omega}+1\leq j\leq N, \text{ if } 1\leq i\leq D_b,\\
		1\leq j\leq N_\Omega, \text{ if } D_{b}+1\leq i\leq D. 
	\end{cases}
\end{align*}
Let $\boldsymbol{\phi}^{(i)}$ be the vector concatenating the functionals $\phi_j^{(i)}$ for fixed $i$ and define
\begin{align}
	\label{defindiline}
	\boldsymbol{\phi}:=(\boldsymbol{\phi}^{(1)}, \dots, \boldsymbol{\phi}^{(D)})\in (\mathcal{U}^*)^{\otimes R}, \text{ where } R=(N-N_\Omega)D_b + N_{\Omega}(D-D_b). 
\end{align}
Next, we define the data vector $\boldsymbol{y}\in \mathbb{R}^N$ by
\begin{align*}
	y_i = \begin{cases}
		f(\boldsymbol{x}_i), \text{ if } i\in \{1,\dots, N_{\Omega}\},\\
		g(\boldsymbol{x}_i), \text{ if } i\in \{N_{\Omega}+1, \dots, N\}. 
	\end{cases}
\end{align*}
Furthermore, we define the nonlinear map
\begin{align*}
	(F([\boldsymbol{\phi}, u])_j:=\begin{cases}
		{P}([\phi_j^{(D_b+1)}, u],\dots, [\phi_j^{(D)}, u]), \text{ for } j = 1,\dots, N_{\Omega},\\
		{B}([\phi_j^{(1)}, u],\dots, [\phi_j^{(D_b)}, u]), \text{ for } j=N_{\Omega}+1,\dots, N. 
	\end{cases}
\end{align*}
Then, given a regularization parameter $\lambda>0$, the GP method \cite{chen2021solving} approximates $u^*$ by the minimizer of the following optimization problem
\begin{align}
	\label{pdeminprob}
	\min\limits_{u\in \mathcal{U}}&  \frac{\lambda}{2}\|u\|_{\mathcal{U}}^2 + \frac{1}{2N}|F([\boldsymbol{\phi}, u] - \boldsymbol{y}|^2. 
\end{align}
The authors in \cite{chen2021solving} show that as $\lambda\rightarrow 0$ and as the fill distance of samples decreases, a minimizer of \eqref{pdeminprob} converges to the unique strong solution of \eqref{pdegrf} (see Sec. 3 of \cite{chen2021solving}). 

\subsection{GP Regressions with Nonlinear Measurements}
\label{subsecGPnl}
In this subsection, we propose a mini-batch method to solve \eqref{pdeminprob}. For ease of presentation, we abstract \eqref{pdeminprob} into a general framework by considering a GP problem with nonlinear observations. More precisely, let $\mathcal{N}=\{1, \dots, N\}$ and  $\boldsymbol{\phi}=(\boldsymbol{\phi}_1, \dots, \boldsymbol{\phi}_N)$, where $\boldsymbol{\phi}_i\in (\mathcal{U}^*)^{\bigotimes N_i}$ for some $N_i\in \mathbb{N}$ and $i\in \mathcal{N}$. For an unknown function $u$ and  each $i\in \mathcal{N}$, there exists a function $f_i$ such that we get a noisy observation $y_i$ on $f_i([\boldsymbol{\phi}_i, u])$. To approximate $u$ based on the nonlinear observations, we  solve 
\begin{align}
	\label{nlprob}
	\min_{u\in \mathcal{U}}\frac{\lambda}{2}\|u\|_{\mathcal{U}}^2 + \frac{1}{2N}\sum_{i=1}^N|f_i([\boldsymbol{\phi}_i, u])-{y}_i|^2,
\end{align}
where  $\lambda>0$. 
Hence, \eqref{nlprob} resembles \eqref{pdeminprob} in such a way that 
$f_i([\boldsymbol{\phi}_i, u])-y_i$  corresponds to the PDE equation satisfied by $u$ at the $i^{\textit{th}}$ sample point. 

As we mentioned before, a typical algorithm using the Cholesky decomposition to compute the inverse of the corresponding covariance matrix suffers the cost of $O(N^3)$. Here, we introduce an algorithm to solve \eqref{nlprob} using mini-batches such that the dimensions of matrices to be inverted depend only on the sizes of mini-batches. 

First, we reformulate a relaxed version of \eqref{nlprob}. More precisely, we introduce slack variables $\boldsymbol{z}=(\boldsymbol{z}_1,\dots,\boldsymbol{z}_N)$, a regularization parameter $\beta>0$, and consider the following regularized optimal recovery problem
\begin{align}
	\label{nltlprob}
	\min_{u\in \mathcal{U}, \boldsymbol{z}\in \mathcal{X}}\varphi(u, \boldsymbol{z}), 
\end{align}
where
\begin{align}
	\label{nldefvph}
	\varphi(u, \boldsymbol{z}) = \frac{\lambda}{2}\|u\|_{\mathcal{U}}^2 + \frac{\beta}{2N}\sum_{i=1}^N|[\boldsymbol{\phi}_i, u]-\boldsymbol{z}_i|^2 + \frac{1}{2N}\sum_{i=1}^N|f_i(\boldsymbol{z}_i)-{y}_i|^2,
\end{align}
$\mathcal{X}$ is a convex compact subset of $\mathbb{R}^{\sum_{i=1}^NN_i}$, $0\in \mathcal{X}$, and $\boldsymbol{y}\in \mathcal{X}$. We further assume that $\mathcal{X}$ is the product space of a set of convex compact subsets $\{\mathcal{X}_i\}^N$, where $\mathcal{X}_i\in \mathbb{R}^{N_i}$.  
\begin{remark}
	\label{rmklubdmnp}
	The restriction of $\boldsymbol{z}$ on $\mathcal{X}$ is not restrictive. Indeed, we can reformulate the unconstrained minimization problem
	\begin{align}
		\label{GPLinMFm1mbct}
		\min_{u\in \mathcal{U}, \boldsymbol{z}\in \mathbb{R}^{\sum_{i=1}^NN_i}}\frac{\lambda}{2}\|u\|_{\mathcal{U}}^2 + \frac{\beta}{2N}\sum_{i=1}^N|[\boldsymbol{\phi}_i, u]-\boldsymbol{z}_i|^2 + \frac{1}{2N}\sum_{i=1}^N|f_i(\boldsymbol{z}_i)-{y}_i|^2, 
	\end{align}
	into the formulation of \eqref{nltlprob} by introducing a large enough compact set $\mathcal{X}$. More precisely, let $(u_\beta, \boldsymbol{z}_\beta)$ be the minimizer of \eqref{GPLinMFm1mbct}. We have 
	\begin{align*}
		\varphi(u_{\beta}, \boldsymbol{z}_\beta) \leq \varphi(0, \boldsymbol{0}) = \frac{1}{2N}\sum_{i=1}^N|f_i(\boldsymbol{0}) - y_i|^2,
	\end{align*}
	which implies that $u_\beta$ and $\boldsymbol{z}_\beta$ are bounded if $\sum_{i=1}^N|f_i(\boldsymbol{0})-y_i|^2<+\infty$. Hence, if $f_i, i=1,\dots, N$ are bounded,  \eqref{GPLinMFm1mbct} and \eqref{nltlprob} are equivalent for large  $\mathcal{X}$. Thus, we mainly focus on \eqref{nltlprob} to take into consideration   situations when $\boldsymbol{z}$ is constrained. 
	
\end{remark}

Next, we introduce a mini-batch version of  $\varphi$ defined in \eqref{nltlprob}. For $\mathcal{I}\subset \mathcal{N}$ with $|\mathcal{I}|=M$, we define
\begin{align}
	\label{nlGPvpI}
	\varphi(u, \boldsymbol{z}; \mathcal{I})=\frac{\lambda}{2}\|u\|_{\mathcal{U}}^2 + \frac{\beta}{2 M}\sum_{i\in \mathcal{I}}|[\boldsymbol{\phi}_i, u] - \boldsymbol{z}_i|^2 + \frac{1}{2M}\sum_{i\in \mathcal{I}}|f_i(\boldsymbol{z}_i) - y_i|^2. 
\end{align}
In particular, when $M=1$ and $\mathcal{I}=\{\xi\}$ for $\xi\in \mathcal{N}$, we denote
\begin{align*}
	\varphi(u, \boldsymbol{z}; \xi):=\varphi(u, \boldsymbol{z}; \{\xi\})=\frac{\lambda}{2}\|u\|_{\mathcal{U}}^2 + \frac{\beta}{2}|[\boldsymbol{\phi}_{\xi}, u] - \boldsymbol{z}_{\xi}|^2 + \frac{1}{2}|f_{\xi}(\boldsymbol{z}_{\xi}) - y_{\xi}|^2. 
\end{align*}

Then, we detail the mini-batch method in Algorithm \ref{alg:1}. Given  $\boldsymbol{z}^1\in \mathcal{X}$ and the iteration number $K$, at the $k^{\textit{th}}$ step, we update $(u^{k+1}, \boldsymbol{z}^{k+1})$ by solving the minimization problem \eqref{alg:1:upts}.
Finally, we uniformly sample $k^*$ from $\{1, \dots, K\}$ and output $(\hat{u}^{k^*}, \hat{\boldsymbol{z}}^{k^*})$ as an approximation to a solution of \eqref{nltlprob}. 
\begin{remark}
	\label{lnfse}
	In Algorithm \ref{alg:1}, we introduce the last step \eqref{alg:1:fnop} for error analysis. 
	In practice, we do not need to solve \eqref{alg:1:fnop} directly since computing \eqref{alg:1:fnop} consumes the same computational cost as calculating \eqref{nlprob}. Instead, if we are interested in the value of $u$ at a test point $\Tilde{\boldsymbol{x}}$, we can take a mini-batch $\Tilde{\mathcal{I}}$ containing training samples at the neighbor of  $\Tilde{\boldsymbol{x}}$, choose $k^*$ and $\rho$ as in Algorithm \ref{alg:1}, compute
	$(\Tilde{u}, \Tilde{\boldsymbol{z}}) = \argmin_{u\in \mathcal{U}, \boldsymbol{z}\in \mathcal{X}} \varphi(u, \boldsymbol{z}; \Tilde{\mathcal{I}}) + \frac{\rho}{2}|\boldsymbol{z} - \boldsymbol{z}^{k^*}|^2$,
	and evaluate $\Tilde{u}$ at $\Tilde{\boldsymbol{x}}$. An alternative way is to include the test points of interest in the training samples and solve \eqref{nltlprob} using the mini-batch method. Then, the latent vector $\boldsymbol{z}^{k^*}$ contains sufficient information at $\tilde{\boldsymbol{x}}$. 
\end{remark}
{\color{black}
\begin{remark}
\label{rmk:pcd}
	Algorithm \ref{alg:1} can be readily extended to incorporate a preconditioned variant. Specifically, let \( Q_k \) be a positive definite matrix associated with the \( k^{\textit{th}} \) iteration. The update step in Equation \eqref{alg:1:upts} is modified as follows:
	\begin{align}
		\label{alg:1:precond:upts}
		(u^{k+1}, \boldsymbol{z}^{k+1}) =& \argmin_{u\in \mathcal{U}, \boldsymbol{z}\in \mathcal{X}} \varphi(u, \boldsymbol{z}; \mathcal{I}_k) + \frac{1}{2}|\boldsymbol{z} - \boldsymbol{z}^k|^2_{Q_k},
	\end{align}
	where the norm \( |\cdot|_{Q_k} \) is defined by \( |y|_{Q_k}^2=y^TQ_k^{-1}y \). This preconditioned approach maintains the theoretical foundations of Algorithm \ref{alg:1}, given that the minimum eigenvalue of $Q_k$ meets the analogous conditions set by \( \gamma_k \) in \eqref{alg:1:upts} for each iteration \( k \). This adaptation offers enhanced flexibility and potential efficiency improvements in various computational scenarios.
\end{remark}
}

\begin{algorithm}
	\caption{The GP Regression With Mini-batches}\label{alg:1}
	$\textbf{Input}: \lambda>0, \beta>0, \boldsymbol{z}^1\in \mathcal{X}, \text{iteration number $K$}$
	
	$\textbf{For } k = 1,\dots,K$
	
	Take a mini-batch $\mathcal{I}_k$ from $\mathcal{N}$, choose a step size $\gamma_k$, and update $\boldsymbol{z}^{k+1}$ by solving 
	\begin{align}
		(u^{k+1}, \boldsymbol{z}^{k+1}) =& \argmin_{u\in \mathcal{U}, \boldsymbol{z}\in \mathcal{X}} \varphi(u, \boldsymbol{z}; \mathcal{I}_k) + \frac{\gamma_k}{2}|\boldsymbol{z} - \boldsymbol{z}^k|^2. \label{alg:1:upts}
	\end{align}
	
	$\textbf{End For}$
	
	Choose $\rho>0$, sample $k^*$ uniformly in $\{1, \dots, K\}$ and compute
	\begin{align}
		(\hat{u}^{k^*}, \hat{\boldsymbol{z}}^{k^*}) =& \argmin_{u\in \mathcal{U}, \boldsymbol{z}\in \mathcal{X}} \varphi(u, \boldsymbol{z}) + \frac{\rho}{2}|\boldsymbol{z} - \boldsymbol{z}^{k^*}|^2. \label{alg:1:fnop}
	\end{align}
	
	$\textbf{Output: } \hat{u}^{k^*}, \hat{\boldsymbol{z}}^{k^*}$. 
	
\end{algorithm}
Though the minimization problem in  \eqref{alg:1:upts} is infinite-dimensional,  we can reformulate \eqref{alg:1:upts} into a finite-dimensional minimization problem by using the framework developed in \cite{smale2005shannon}. The following lemma resembles Lemma 3.3 in \cite{meng2023sparse} and establishes properties of a sampling operator, which are crucial for us to derive a representer formula for \eqref{alg:1:upts}. 
We omit the proof since the arguments are the same as those for Lemma 3.3 of \cite{meng2023sparse}. For simplicity, we denote by $I$ the identity map or the identity matrix, whose meaning is easily recognizable from the context. 

\begin{lemma}
	\label{solm}
	Let $\mathcal{N}=\{1, \dots, N\}$ and $\mathcal{I}\subset \mathcal{N}$. Let $\boldsymbol{\phi}$ be as in \eqref{nlprob} and let $\boldsymbol{\phi}_{\mathcal{I}}$ be the subset of $\boldsymbol{\phi}$ indexed by $\mathcal{I}$. Denote by $\ell^2(\mathcal{I})$ the set of sequences $\boldsymbol{a}=(a_{i})_{i \in \mathcal{I}}$ with the inner product $\langle \boldsymbol{a}, \boldsymbol{b}\rangle=\sum_{i\in \mathcal{I}}a_{i}b_{i}$. 
	Define the sampling operator $S_{\mathcal{I}}:\mathcal{U}\mapsto \ell^2(\mathcal{I})$ by $
		S_{\mathcal{I}}u=[\boldsymbol{\phi}_{\mathcal{I}}, u], \forall u\in \mathcal{U}$.
	Let $S_{\mathcal{I}}^*$ be the adjoint of $S_{\mathcal{I}}$.  Then, for each $c\in \ell^2(\mathcal{I})$, 
		$S_{\mathcal{I}}^*c = c^T\mathcal{K}\boldsymbol{\phi}_{\mathcal{I}}, \forall c\in \ell^2(\mathcal{I})$. 
	Besides, for any $\gamma>0$, ${S}_{\mathcal{I}}^*{S}_{\mathcal{I}}+\gamma I$ and ${S}_{\mathcal{I}}{S}_{\mathcal{I}}^*+\gamma I$ are bijective. Meanwhile, 
	for any $c\in \ell^2(\mathcal{I})$, we have
	\begin{align}
		\label{eqre0}
		({S}_{\mathcal{I}}^*{S}_{\mathcal{I}}+\gamma I)^{-1}{S}_{\mathcal{I}}^*c = (\mathcal{K}\boldsymbol{\phi}_{\mathcal{I}})^T(\gamma I + \mathcal{K}(\boldsymbol{\phi}_{\mathcal{I}}, \boldsymbol{\phi}_{\mathcal{I}}))^{-1}c. 
	\end{align}
	Furthermore, 
	\begin{align}
		\label{bdImus}
		I-S_{\mathcal{I}}(S_{\mathcal{I}}^*S_{\mathcal{I}}+\gamma I)^{-1}S_{\mathcal{I}}^*=\gamma(\mathcal{K}(\boldsymbol{\phi}_{\mathcal{I}}, \boldsymbol{\phi}_{\mathcal{I}})+\gamma I)^{-1}. 
	\end{align}
\end{lemma}

The next lemma provides a representer formula for \eqref{alg:1:upts}. The arguments are similar to those for Theorem 3.4 of \cite{meng2023sparse}, which we postpone to Appendix \ref{secAppendix}. We recall that for an index set $\mathcal{I}$ and a vector $\boldsymbol{v}$,  $\boldsymbol{v}_{\mathcal{I}}$ represents the vector consisting of elements of $\boldsymbol{v}$ indexed by $\mathcal{I}$. 
\begin{lemma}
	\label{nlreprelm}
	Let $\mathcal{N}=\{1, \dots, N\}$, let $\mathcal{I}\subset \mathcal{N}$, $|\mathcal{I}|=M$,  and let $\varphi(\cdot, \cdot;\mathcal{I}):\mathcal{U}\times\mathcal{X}\mapsto \mathbb{R}$ be defined as in \eqref{nlGPvpI}. Given $\gamma\geq 0$ and $\Bar{\boldsymbol{z}}\in \mathcal{X}$, we define 
	\begin{align}
		\label{nlhproxpr}
		(u^{\dagger}, {\boldsymbol{z}}^{\dagger})=\argmin_{u\in \mathcal{U}, \boldsymbol{z}\in \mathcal{X}}\varphi(u, \boldsymbol{z}; \mathcal{I}) +  \frac{\gamma}{2}|\boldsymbol{z}-\bar{\boldsymbol{z}}|^2. 
	\end{align}
	Let $\boldsymbol{\phi}$ be as in \eqref{nlprob} and let $\boldsymbol{\phi}_{\mathcal{I}}$ be the subset of $\boldsymbol{\phi}$ indexed by $\mathcal{I}$.  Then,
	\begin{align}
		\label{nluzprli}
		u^{\dagger}=\mathcal{K}(\boldsymbol{\phi}_{\mathcal{I}})\bigg(\mathcal{K}(\boldsymbol{\phi}_{\mathcal{I}}, \boldsymbol{\phi}_{\mathcal{I}}) + \frac{\lambda M}{\beta}{I} \bigg)^{-1}\boldsymbol{z}^{\dagger}_{\mathcal{I}}, 
	\end{align}
	where ${I}$ is the identity matrix and  $\boldsymbol{z}^\dagger=(\bar{\boldsymbol{z}}_{\mathcal{N}\backslash\mathcal{I}}, \boldsymbol{z}^\dagger_{\mathcal{I}})$   is a solution to
	\begin{align}
		\label{nlmzp}
		\min_{\boldsymbol{z}\in \mathcal{X}}\frac{\lambda}{2}\boldsymbol{z}^T_{\mathcal{I}}\bigg(\mathcal{K}(\boldsymbol{\phi}_{\mathcal{I}}, \boldsymbol{\phi}_{\mathcal{I}}) + \frac{\lambda M}{\beta}I\bigg)^{-1}\boldsymbol{z}_{\mathcal{I}} + \frac{1}{2M}\sum_{i\in \mathcal{I}}|f_i(z_i) - y_i|^2 +  \frac{\gamma}{2}|\boldsymbol{z}-\Bar{\boldsymbol{z}}|^2. 
	\end{align}
\end{lemma}

The above lemma establishes the solvability of \eqref{alg:1:upts} and implies that in each step of Algorithm \ref{alg:1},  the sizes of the matrices to be inverted depend only on the batch size $M$. Thus, the computational cost is reduced to $O(M^3)$ if we use the standard Cholesky decomposition. In numerical implementations of solving PDEs, an alternative way to deal with the PDE constraints in \eqref{nlGPvpI} at each step is to use the technique of eliminating variables (see Subsection 3.3 of \cite{chen2021solving} for details), which further reduces the dimensions of unknown variables. 

\begin{remark}
\label{rmknugget}
In numerical experiments, we replace the term $\lambda M I / \beta$ by a nugget $\eta \mathcal{R}$, where $\eta >0$ is a small real number and $\mathcal{R}$ is a block diagonal matrix formulated using the method described in Sec. 3.4 of \cite{chen2021solving}. Using block diagonal matrices lets us choose smaller nuggets and get more accurate results. 
\end{remark}

\section{The Convergence Analysis}
\label{secConverge}
In this section, we study the convergence of Algorithm \ref{alg:1}.  According to  \eqref{nluzprli},  at the $k^{\textit{th}}$ step of Algorithm \ref{alg:1}, $u^{k+1}$ is expressed as the linear combinations of $\mathcal{K}(\boldsymbol{\phi}_{\mathcal{I}})$, where $\boldsymbol{\phi}_{\mathcal{I}}$   only contains linear operators corresponding to a mini-batch, which is not expressive. Thus, at the final $K^{\textit{th}}$ step, we do not expect that $u^{K+1}$ approximates a minimizer of \eqref{nltlprob} very well. Instead, we study the accuracy of the sequence $(\boldsymbol{z}^k)_{k=1}^K$ in  Algorithm \ref{alg:1}, which contains the approximated values of the linear operators of the true solution, and we study the convergence of $(\hat{u}^{k^*}, \hat{\boldsymbol{z}}^{k^*})$ in \eqref{alg:1:fnop}, which uses the information of all the linear operators. 

We first summarize some basic notions of convex analysis in Subsection \ref{subsecPrere}. Then, we introduce the stability of Algorithm \ref{alg:1} in Subsection \ref{subsecStab}. The convergence analysis follows in Subsection \ref{subsecConv}. The arguments  are based on the framework of \cite{davis2019stochastic} and stability analysis \cite{bousquet2002stability}. Similar arguments are used in \cite{deng2021minibatch} to analyze the convergence of algorithms for general weakly convex optimization problems. 

\subsection{Prerequisites}
\label{subsecPrere}
For a non-empty, compact, and convex set  $\mathcal{C}\subset \mathbb{R}^d$ for some $d\in \mathbb{N}$, the normal cone to $\mathcal{C}$ at $\boldsymbol{x}\in \mathcal{C}$ is defined by $N_{\mathcal{C}}(\boldsymbol{x}) = \{\boldsymbol{y}\in \mathbb{R}^d| \boldsymbol{y}\cdot (\boldsymbol{c} - \boldsymbol{x})\leq 0, \forall \boldsymbol{c}\in \mathcal{C} \}$. 
Given $\gamma\geq 0$, a function $\omega: \mathbb{R}^d\mapsto\mathbb{R}$ is $\gamma$-weakly convex if the map $\boldsymbol{x}\mapsto \omega(\boldsymbol{x})+\frac{\gamma}{2}|\boldsymbol{x}|^2$ is convex and  $\omega$ is $\gamma$-strongly convex if $\boldsymbol{x}\mapsto \omega(\boldsymbol{x})-\frac{\gamma}{2}|\boldsymbol{x}|^2$ is convex. We recall that the proximal map of $\omega$ is given by
	$\operatorname{prox}_{\gamma \omega}(\boldsymbol{x})=\argmin_{\boldsymbol{y}}\Big\{\omega(\boldsymbol{y})+\frac{1}{2\gamma}|\boldsymbol{y}-\boldsymbol{x}|^2 \Big\}$,
and that the Moreau envelop \cite{moreau1965proximite} is defined as  
	$\omega_{\gamma}(\boldsymbol{x})=\min_{\boldsymbol{y}}\Big\{\omega(\boldsymbol{y})+\frac{1}{2\gamma}|\boldsymbol{y}-\boldsymbol{x}|^2 \Big\}$. 
Standard results (\cite{moreau1965proximite} and  \cite[Theorem 31.5] {rockafellar1970convex}) implies that  if $\omega$ is $\rho$-weakly convex and $\gamma < \rho^{-1}$, the envelope $\omega_{\gamma}$ is $C^1$ smooth with the gradient given by
	$\nabla \omega_{\gamma}(\boldsymbol{x})=\gamma^{-1}(\boldsymbol{x}-\operatorname{prox}_{\gamma \omega}(\boldsymbol{x}))$. 
Let $\hat{\boldsymbol{x}}=\operatorname{prox}_{\gamma \omega}(\boldsymbol{x})$. Then, the definition of the Moreau envelope implies that 
	$|\partial \omega(\hat{\boldsymbol{x}}) + N_{\mathcal{C}}(\hat{\boldsymbol{x}})| \leq |\nabla \omega_{\gamma}(\boldsymbol{x})|$. Hence, a point $\boldsymbol{x}$ with a small gradient $\nabla \omega_{\gamma}(\boldsymbol{x})$ lies in the neighbor of a nearly stationary point $\hat{\boldsymbol{x}}$. 

\begin{comment}
	The key construction we rely on is the Moreau envelope defined in \eqref{nlmorlevp}. Lemma \ref{nlreprelm} shows that solving \eqref{nltlprob} is equivalent to minimizing $\psi$ defined by \eqref{nldefpsi}. For some $\rho>0$, let $\varphi_{1/\rho}$ be defined as in \eqref{nlmorlevp}. Then, the authors in \cite{davis2019stochastic}  show that 
	\begin{align}
		\label{lstbcd}
		|\partial \psi(\hat{\boldsymbol{z}}) + N_{\mathcal{X}}(\hat{\boldsymbol{z}})| \leq |\nabla \varphi_{1/\rho}(\boldsymbol{z})|.
	\end{align}
	Thus, \eqref{nlgcmf} and \eqref{lstbcd} imply that a point $\boldsymbol{z}$ with a small gradient $|\nabla \varphi_{1/\rho}(\boldsymbol{z})|$ lies in the neighbor of a nearly stationary point $\hat{\boldsymbol{z}}$. In the following of this section, we seek to bound $|\nabla \varphi_{1/\rho}(\hat{\boldsymbol{z}}^k)|$ where we recall that $(\hat{u}^k, \hat{\boldsymbol{z}}^k)= \widetilde{\operatorname{prox}}_{\varphi/\rho}(\boldsymbol{z}^k)$. 
\end{comment}

The following lemma establishes a crucial inequality for the convergence analysis.
\begin{lemma}
	\label{wc3plemma}
	Assume that $\vartheta(u, \boldsymbol{z})\in \mathcal{U}\times \mathcal{X}$ is $\gamma$-strongly convex in $u$ and $\varpi$-weakly convex in $\boldsymbol{z}$. Let $\tau>\varpi$ and
		$(\hat{u}, \hat{\boldsymbol{z}})=\argmin_{u\in \mathcal{U}, \boldsymbol{z}\in \mathcal{X}}\{\vartheta(u, \boldsymbol{z}) + \frac{\tau}{2}|\boldsymbol{z}-\Bar{\boldsymbol{z}}|^2\}$.
	Then, for any $(u, \boldsymbol{z})\in \mathcal{U}\times\mathcal{X}$, we have
	\begin{align}
		\label{wc3plemmafml}
		\vartheta(\hat{u}, \hat{\boldsymbol{z}}) + \frac{\tau}{2}|\hat{\boldsymbol{z}}-\Bar{\boldsymbol{z}}|^2 \leq \vartheta(u, \boldsymbol{z}) + \frac{\tau}{2}|\boldsymbol{z}-\Bar{\boldsymbol{z}}|^2 - \frac{\gamma}{2}\|u - \hat{u}\|^2_{\mathcal{U}} - \frac{\tau - \varpi}{2}|\boldsymbol{z} - \hat{\boldsymbol{z}}|^2. 
	\end{align}
\end{lemma}
\begin{proof}
	Let $(\hat{u}, \hat{\boldsymbol{z}})=\argmin_{u\in \mathcal{U}, \boldsymbol{z}\in \mathcal{X}}\{\vartheta(u, \boldsymbol{z}) + \frac{\tau}{2}|\boldsymbol{z}-\Bar{\boldsymbol{z}}|^2\}$ and $\hat{g}\in \partial\vartheta(\hat{u}, \hat{\boldsymbol{z}})$. Then,  $
		\boldsymbol{0}\in \partial\vartheta(\hat{u}, \hat{\boldsymbol{z}}) + \tau\begin{pmatrix}0\\ \hat{\boldsymbol{z}}-\bar{\boldsymbol{z}}\end{pmatrix}+\begin{pmatrix}0 \\ {N}_{\mathcal{X}}(\hat{\boldsymbol{z}})\end{pmatrix}$.
	Thus, for any $(u, \boldsymbol{z})\in \mathcal{U}\times\mathcal{X}$, we have
	\begin{align}
		\label{GMLinOpt1}
		\frac{1}{\tau}\bigg\langle \hat{g}, \begin{pmatrix}\hat{u}-u\\ \hat{\boldsymbol{z}}-\boldsymbol{z}\end{pmatrix}\bigg\rangle\leq&  \langle \hat{\boldsymbol{z}} - \bar{\boldsymbol{z}}, \boldsymbol{z} - \hat{\boldsymbol{z}}\rangle=   \frac{1}{2}|\boldsymbol{z}-\bar{\boldsymbol{z}}|^2-\frac{1}{2}|\hat{\boldsymbol{z}}-\bar{\boldsymbol{z}}|^2-\frac{1}{2}|\boldsymbol{z}-\hat{\boldsymbol{z}}|^2.
	\end{align}
	Since $\vartheta(\cdot, \cdot)$ is $\gamma$-strongly convex in $u$ and $\varpi$-weakly convex in $\boldsymbol{z}$, we have
	\begin{align}
		\label{GMLinOptm2}
		\bigg\langle \hat{g}, \begin{pmatrix}\hat{u}-u\\ \hat{\boldsymbol{z}}-\boldsymbol{z}\end{pmatrix}\bigg\rangle \geq \vartheta(\hat{u}, \hat{\boldsymbol{z}}) - \vartheta(u, \boldsymbol{z}) + \frac{\gamma}{2}\|u-\hat{u}\|_{\mathcal{U}}^2 - \frac{\varpi}{2}|\boldsymbol{z} - \hat{\boldsymbol{z}}|^2.  
	\end{align}
	Thus, combining \eqref{GMLinOpt1} and \eqref{GMLinOptm2}, we conclude \eqref{wc3plemmafml}. 
\end{proof}
\begin{comment}
	\begin{lemma}
		\label{GPLin3plm}
		Let $I$ be the indices of a minibatch. Given $\overline{z}\in \mathbb{R}^N$. Define
		\begin{align}
			\label{uhzhmn}
			(\hat{u}, \hat{\boldsymbol{z}}) = \argmin_{u\in \mathcal{U}, \boldsymbol{z}\in \mathcal{X}}\varphi(u, \boldsymbol{z}; I) +\frac{\gamma}{2}|\boldsymbol{z}-\overline{\boldsymbol{z}}|^2. 
		\end{align}
		Then, for any $(u, \boldsymbol{z})\in \mathcal{U}\times\mathcal{X}$, we have
		\begin{align}
			\label{GMLinOpt3plmeq}
			\frac{\lambda}{2\gamma}\|u-\hat{u}\|_{\mathcal{U}}^2+\frac{1}{2}|\boldsymbol{z}-\hat{\boldsymbol{z}}|^2 \leq&  \frac{1}{2}|\boldsymbol{z}-\overline{\boldsymbol{z}}|^2 -\frac{1}{2}|\hat{\boldsymbol{z}}-\overline{\boldsymbol{z}}|^2-\frac{1}{\gamma}(\varphi(\hat{u}, \hat{\boldsymbol{z}};I)-\varphi(\boldsymbol{u}, \boldsymbol{z};I)), \\
			\varphi(u, \boldsymbol{z}; I) - \varphi(\hat{u}, \hat{\boldsymbol{z}}; I) \geq& \frac{\gamma}{2}|\hat{\boldsymbol{z}}-\bar{\boldsymbol{z}}|^2 + \gamma\langle \hat{\boldsymbol{z}} - \bar{\boldsymbol{z}}, \bar{\boldsymbol{z}} - \boldsymbol{z}\rangle.  
		\end{align}
	\end{lemma}
\end{comment}

\subsection{Stability}
\label{subsecStab}
In this subsection, we define the notion of stability for \eqref{alg:1:upts} in a general setting. Let ${\mathcal{I}} = \{\xi_1, \dots, \xi_{M}\}$ be a set of i.i.d. samples from ${\mathcal{N}}=\{1, \dots, N\}$, where $M, N \in \mathbb{N}$ and $M \leq N$. Let $i\in [1, M]$.  We replace $\xi_i$ with an i.i.d. copy $\xi_i'$ and define ${\mathcal{I}}_{(i)}=\{\xi_1, \dots, \xi_{i-1}, \xi'_{i}, \xi_{i+1}, \dots, \xi_M\}$. Denote ${\mathcal{I}}' = \{\xi'_1, \dots, \xi'_{M}\}$. Let $\xi$ be a random sample from $\mathcal{N}$, let $\psi(\cdot, \cdot; \xi):\mathcal{U}\times\mathcal{X}\mapsto\mathbb{R}$ be a stochastic model,  and let $\psi(\cdot, \cdot; {\mathcal{I}}) = \frac{1}{M}\sum_{i=1}^M\psi(\cdot, \cdot; \xi_i)$. Given $\gamma>0$, we define an operator $\widetilde{\operatorname{prox}}_{\psi/\gamma}^{\mathcal{I}}$ by
\begin{align}
	\label{defhprox}
	\widetilde{\operatorname{prox}}_{\psi/\gamma}^{\mathcal{I}} (\bar{\boldsymbol{z}})= \argmin_{u\in \mathcal{U}, \boldsymbol{z}\in \mathcal{X}}\psi(u, \boldsymbol{z}; {\mathcal{I}}) + \frac{\gamma}{2}|\boldsymbol{z}-\bar{\boldsymbol{z}}|^2, \forall \bar{\boldsymbol{z}}\in \mathcal{X}. 
\end{align}
Given $\bar{\boldsymbol{z}}\in \mathcal{X}$,  let $(\hat{u}_{\mathcal{I}}, \hat{\boldsymbol{z}}_{\mathcal{I}})=\widetilde{\operatorname{prox}}_{\psi/\gamma}^{\mathcal{I}}(\bar{\boldsymbol{z}})$. Then, we say that the operator $\widetilde{\operatorname{prox}}_{\psi/\gamma}^{\mathcal{I}}$ is $\epsilon$-stable for some $\epsilon>0$ if $|E_{{\mathcal{I}}, {\mathcal{I}}', i}[ \psi(\hat{u}_{{\mathcal{I}}}, \hat{\boldsymbol{z}}_{\mathcal{I}}; \xi_i') - \psi(\hat{u}_{{\mathcal{I}}_{(i)}}, \hat{\boldsymbol{z}}_{{\mathcal{I}}_{(i)}}; \xi_i') ]|\leq \epsilon$. The stability measures the accuracy of  $\widetilde{\operatorname{prox}}_{\psi/\gamma}^{\mathcal{I}}$ against the perturbation of one sample. 
The next lemma shows that if $\widetilde{\operatorname{prox}}_{\psi/\gamma}^{\mathcal{I}}$ is $\epsilon$-stable, the approximation of $E_{\xi}[\psi(\hat{u}_{\mathcal{I}}, \hat{\boldsymbol{z}}_{\mathcal{I}};\xi)]$ using a mini-batch is bounded by $\epsilon$ on expectation. The arguments are similar to those of Lemma 11 of \cite{bousquet2002stability} and Theorem A.3 of \cite{deng2021minibatch}.

\begin{lemma}
	\label{mmberdb}
	Let $\xi$ be a uniform sample from $\mathcal{N}$. For a stochastic model $\psi(\cdot, \cdot;\xi):\mathcal{U}\times\mathcal{X}\mapsto\mathbb{R}$, let $\widetilde{\operatorname{prox}}_{\psi/\gamma}^{\mathcal{I}}$ be as in \eqref{defhprox} given $\gamma>0$. 
	Suppose that $\widetilde{\operatorname{prox}}_{\psi/\gamma}^{\mathcal{I}}$ is $\epsilon$-stable. Let $(\hat{u}, \hat{\boldsymbol{z}}) = \widetilde{\operatorname{prox}}^{\mathcal{I}}_{\psi/\gamma}(\boldsymbol{z})$ for $\boldsymbol{z}\in \mathcal{X}$. Then, we have
	\begin{align}
		\label{l23eq00}
		|E_{{\mathcal{I}}}[\psi(\hat{u}, \hat{\boldsymbol{z}}; {\mathcal{I}}) - E_{\xi}[\psi(\hat{u}, \hat{\boldsymbol{z}}; \xi)]]|\leq \epsilon. 
	\end{align}
\end{lemma}
\begin{proof}
	Let ${\mathcal{I}}_{(i)}$ and  ${\mathcal{I}}'$ be defined as in the beginning of this subsection. 
	Given $\bar{\boldsymbol{z}}\in \mathcal{X}$, let
		$(\hat{u}, \hat{\boldsymbol{z}}) = \widetilde{\operatorname{prox}}_{\psi/\gamma}^{\mathcal{I}}(\bar{\boldsymbol{z}})$ and 
		$(\hat{u}_{(i)}, \hat{\boldsymbol{z}}_{(i)}) = \widetilde{\operatorname{prox}}_{\psi/\gamma}^{{\mathcal{I}}_{(i)}}(\bar{\boldsymbol{z}})$. 
	By the independence between ${\mathcal{I}}$ and ${\mathcal{I}}'$, we have 
	\begin{align}
		\label{l23eq0}
		E_{{\mathcal{I}}}[\psi(\hat{u}, \hat{\boldsymbol{z}}; {\mathcal{I}})] = \frac{1}{M}\sum_{i=1}^ME_{\mathcal{I}}[\psi(\hat{u}, \hat{\boldsymbol{z}}; \xi_i)]=\frac{1}{M}\sum_{i=1}^ME_{{\mathcal{I}},\xi_i'}[\psi(\hat{u}_{(i)}, \hat{\boldsymbol{z}}_{(i)}; \xi_i')].
	\end{align}
	Let $\xi$ and $\xi'$ be independent samples from ${\mathcal{N}}$. Then,   $E_\xi[\psi(\hat{u}, \hat{\boldsymbol{z}}; \xi)] = E_{\xi'}[\psi(\hat{u}, \hat{\boldsymbol{z}}; \xi')] = \frac{1}{M}\sum_{i=1}^ME_{\xi_i'}[\psi(\hat{u}, \hat{\boldsymbol{z}}; \xi_i')]$.  
	Thus,  the prior equation and  \eqref{l23eq0} yield
	\begin{align}
		|E_{{\mathcal{I}}}[\psi(\hat{u}, \hat{\boldsymbol{z}}; {\mathcal{I}}) - E_{\xi}[\psi(\hat{u}, \hat{\boldsymbol{z}}; \xi)] ]|
		=& \bigg|\frac{1}{M}\sum_{i=1}^ME_{{\mathcal{I}}, \xi'_i}[\psi(\hat{u}_{(i)}, \hat{\boldsymbol{z}}_{(i)};\xi_i') - \psi(\hat{u}, \hat{\boldsymbol{z}}; \xi_i')]\bigg|\nonumber\\
		=& |E_{{\mathcal{I}}, {\mathcal{I}}', i}[\psi(\hat{u}_{(i)}, \hat{\boldsymbol{z}}_{(i)}; \xi_i') - \psi(\hat{u}, \hat{\boldsymbol{z}}; \xi_i')]\bigg|\leq \epsilon, \nonumber
	\end{align}
	where the last inequality results from the stability of $\widetilde{\operatorname{prox}}_{\psi/\gamma}^{\mathcal{I}}$. Hence, \eqref{l23eq00}  holds. 
\end{proof}
\subsection{Convergence}
\label{subsecConv}
In this subsection, we study the convergence of Algorithm \ref{alg:1}.  We first assume that all the linear operators in \eqref{nldefvph} are  bounded. 
\begin{hyp}
	\label{hypnlphibd}
	Let $\{\boldsymbol{\phi}_i\}_{i=1}^N$  be the linear operators in \eqref{nldefvph}. Assume that there exists a constant $L>0$ such that  $\|\boldsymbol{\phi}_i\|_{\mathcal{U}^*}\leq L, \forall i\in {1, \dots, N}$. 
\end{hyp}
\begin{remark}
	The boundedness imposed on the linear operators $\{\boldsymbol{\phi}_i\}$ is not restrictive if the kernel $\kappa$ associated with the space $\mathcal{U}$ is smooth enough. Indeed, let {\color{black} $\phi=\delta_{\boldsymbol{x}}\circ \Delta$}, where {\color{black}$\boldsymbol{x}\in \Omega$}. Then, there exists a constant $C$ such that
	{\color{black}
	\begin{align*}
		\|\phi\|_{\mathcal{U}^*} = \sup_{u\in \mathcal{U}}\frac{[\phi, u]}{\|u\|_{\mathcal{U}}} = \sup_{\Delta u\in \mathcal{U}}\frac{u(\boldsymbol{x})}{\|u\|_{\mathcal{U}}} = \sup_{u\in \mathcal{U}}\frac{\langle u, \Delta \kappa(\cdot, \boldsymbol{x})\rangle}{\|u\|_{\mathcal{U}}}\leq \|\Delta\kappa(\cdot, \boldsymbol{x})\|_{\mathcal{U}}\leq C,
	\end{align*}
}where we use the reproducing property in the third equality, the Cauchy--Schwarz inequality in the first inequality, and the {\color{black}regularity} of the kernel in the final inequality. Similar arguments hold for $\phi$ to be {\color{black}other} differential operators. 
\end{remark}

Furthermore, we impose the following assumption on the regularity of $f_i$ in $\mathcal{X}$. 
\begin{hyp}
	\label{hypf}
	Let $\{f_i\}_{i=1}^N$ be functions in \eqref{nldefvph}. 
	For any $i\in \{1,\dots,N\}$, we assume that $f_i$ is continuous differentiable. Moreover, there exist constants $U_f$ and $\mathcal{H}_f$ such that $\sup_{\boldsymbol{x}\in \mathcal{X}}|f_i(\boldsymbol{x})-{y}_i|\leq U_f$ for $i \in \{1, \dots, N\}$ and 
	\begin{align}
		\label{condfi}
		|f_i(\boldsymbol{w}) - f_i(\boldsymbol{x}) - \langle \nabla f_i(\boldsymbol{x}), \boldsymbol{w} - \boldsymbol{x}\rangle | \leq \frac{\mathcal{H}_f}{2}|\boldsymbol{w}-\boldsymbol{x}|^2, \forall \boldsymbol{x}, \boldsymbol{w}\in \mathcal{X}.   
	\end{align}
\end{hyp}

With the boundedness of the second derivative of $f_i$ at hand, we are able to show that $\varphi(\cdot, \cdot)$ in \eqref{nldefvph} and $\varphi(\cdot, \cdot; \mathcal{I})$ in \eqref{nlGPvpI} are weakly convex. The proof of the following lemma is similar to the argument of Lemma 4.2 of \cite{drusvyatskiy2019efficiency}. Recall that $\mathcal{X}$ is a convex compact set and is the product space of a set of convex compact sets $\{\mathcal{X}_i\}_{i=1}^N$.  
\begin{lemma}
	\label{wcvxt}
	Let $\{f_i\}_{i=1}^N$ be the collection of functions in \eqref{nldefvph}. 
	Suppose that Assumption \ref{hypf} holds. Let $U_f$ and $\mathcal{H}_f$ be the coefficients in Assumption \ref{hypf}, and let $\mu=U_f\mathcal{H}_f$. For $i\in \mathcal{N}$, define $h_i(\boldsymbol{z}_i)=\frac{1}{2}|f_i(\boldsymbol{z}_i)-{y}_i|^2$ for $\boldsymbol{z}_i\in \mathcal{X}_i$. Then, $h_i$ is $\mu$-weakly convex in $\boldsymbol{z}$. Moreover, for any set $\mathcal{I}\subset \mathcal{N}$ with $M$ indices, let $h(\boldsymbol{z}) = \frac{1}{M}\sum_{i\in \mathcal{I}}h_i(\boldsymbol{z}_i)$ for $\boldsymbol{z}=(\boldsymbol{z}_1,\dots,\boldsymbol{z}_N)\in\mathcal{X}$. Then, $h$ is $\mu$-weakly convex in $\boldsymbol{z}$. 
\end{lemma}
\begin{proof}
	For $i\in \mathcal{N}$, let $\boldsymbol{z}_i\in \mathcal{X}_i$ and $\boldsymbol{w}_i\in \mathcal{X}_i$.  Using Assumption \ref{hypf}, we obtain
	\begin{align}
		h_i(\boldsymbol{z}_i) =& \frac{1}{2}|f_i(\boldsymbol{z}_i) - {y}_i|^2 \geq \frac{1}{2}|f_i(\boldsymbol{w}_i) - {y}_i|^2 + \langle f_i(\boldsymbol{w}_i)-{y}_i, f_i(\boldsymbol{z}_i )- f_i(\boldsymbol{w}_i)\rangle\nonumber \\
		=& \frac{1}{2}|f_i(\boldsymbol{w}_i) - {y}_i|^2 + \langle f_i(\boldsymbol{w}_i) - {y}_i, \nabla f_i(\boldsymbol{w}_i)(\boldsymbol{z}_i - \boldsymbol{w}_i)\rangle \nonumber\\
		& + \langle f_i(\boldsymbol{w}_i) - {y}_i, f(\boldsymbol{z}_i) - f_i(\boldsymbol{w}_i) - \nabla f_i(\boldsymbol{w}_i)(\boldsymbol{z}_i - \boldsymbol{w}_i)\rangle\nonumber\\
		\geq & h_i(\boldsymbol{w}_i) + \langle (\nabla f_i(\boldsymbol{w}_i))^T(f_i(\boldsymbol{w}_i) - {y}_i), \boldsymbol{z}_i - \boldsymbol{w}_i\rangle - \frac{\mu}{2}|\boldsymbol{z}_i - \boldsymbol{w}_i|^2,\nonumber
	\end{align}
	which implies that $h_i$ is $\mu$-weakly convex. Thus, for $\boldsymbol{z}, \boldsymbol{w}\in \mathcal{X}$, we have
	\begin{align}
		h(\boldsymbol{z})  \geq & \frac{1}{M}\sum_{i\in \mathcal{I}} h_i(\boldsymbol{w}_i) + \langle (\nabla f_i(\boldsymbol{w}_i))^T(f_i(\boldsymbol{w}_i) - {y}_i), \boldsymbol{z}_i - \boldsymbol{w}_i\rangle - \frac{\mu}{2}|\boldsymbol{z}_i - \boldsymbol{w}_i|^2 \nonumber\\
		\geq & h(\boldsymbol{w}) + \langle \nabla h(\boldsymbol{w}), \boldsymbol{z} - \boldsymbol{w}\rangle - \frac{\mu}{2}|\boldsymbol{z}-\boldsymbol{w}|^2. \nonumber
	\end{align}
	Hence, $h$ is also $\mu$-weakly convex. 
\end{proof}

Lemma \ref{wcvxt} directly implies the weakly convexity of $\varphi(\cdot, \cdot; \mathcal{I})$ defined in \eqref{nlGPvpI}, which is described in the following corollary. 
\begin{corollary}
	\label{corovarphicv}
	Suppose that Assumption \ref{hypf} holds. Let $\mathcal{I}\subset \mathcal{N}$ and let $\mu$ be the constant in Lemma \ref{wcvxt}. Then, $\varphi(\cdot, \cdot; \mathcal{I})$ defined in \eqref{nlGPvpI} is $\lambda$-strongly convex in $u$ and $\mu$-weakly convex in $\boldsymbol{z}$. 
\end{corollary}
\begin{proof}
	The claim follows from Assumption \ref{hypf}, Lemma \ref{wcvxt}, and the strongly convexity of $\varphi(\cdot, \cdot;\mathcal{I})$ in $u$.  
\end{proof}

Next, we use the stability analysis and the weakly convexity property to study the convergence of Algorithm \ref{alg:1}. The following lemma establishes the stability for each iteration of Algorithm \ref{alg:1}.
The proof appears in Appendix \ref{secAppendix}. 

\begin{lemma}
	\label{nlstles}
	Suppose that Assumptions \ref{hypnlphibd} and \ref{hypf} hold. Let $\mathcal{I} = \{\xi_1, \dots, \xi_{M}\}$ be the i.i.d. samples from $\mathcal{N}=\{1, \dots, N\}$, where $M, N \in \mathbb{N}$ and $M \leq N$. Given $\bar{\boldsymbol{z}}\in \mathcal{X}$ and $\gamma>\mu$, where $\mu$ is the constant in Lemma \ref{wcvxt}, we define
	\begin{align} 
		\label{uhzhmn}
		(\hat{u}, \hat{\boldsymbol{z}})=\widetilde{\operatorname{prox}}_{\varphi/\gamma}^{\mathcal{I}}(\bar{\boldsymbol{z}})=\argmin_{u\in \mathcal{U}, \boldsymbol{z}\in \mathcal{X}}\varphi(u, \boldsymbol{z}; {\mathcal{I}}) + \frac{\gamma}{2}|\boldsymbol{z} - \bar{\boldsymbol{z}}|^2.
	\end{align}
	Let $r$ be the diameter of $\mathcal{X}$ and let $U_f$ be given in Assumption \ref{hypf}. Then, there exists a constant $C$ depending on $\beta, r, \lambda, L$, and $U_f$,  such that  $\widetilde{\operatorname{prox}}_{\varphi/\gamma}^{\mathcal{I}}$ is $\epsilon$-stable with 
	\begin{align}
		\label{nonepsdef}
		\epsilon = \frac{C}{M\min\{\lambda, \gamma-\mu\}}.
	\end{align}
\end{lemma}
{\color{black}
\begin{remark}
\label{rmk:stcerr}
Lemma \ref{nlstles} demonstrates the critical distinction between the mini-batch method delineated in this paper and those explored in \cite{duchi2018stochastic, davis2019stochastic, deng2021minibatch}. In our approach, the regularization parameter $\lambda$ is typically small, rendering $\min\{\lambda, \gamma-\mu\}$ equivalent to $\lambda$, a fixed parameter that is not dependent on the step size. This aspect is  significantly different from the methodologies in \cite{duchi2018stochastic, davis2019stochastic, deng2021minibatch}, where $\gamma$ increasing with the number of iterations enables a reduction in statistical error. A notable difference in our method is the omission of a proximal step over $u$ in \eqref{alg:1:upts}. Incorporating the proximal step over $u$ would necessitate the retention of all mini-batch samples and the dependency of the minimizer $u^k$ at the $k^{\textit{th}}$ step on the formulations of preceding minimizers, $u^1,\dots,u^{k-1}$, thereby substantially elevating memory requirements. Addressing the challenge of reducing statistical error while managing memory efficiency is a subject for future research.
\end{remark}
}

Next, we study the descent property of Algorithm \ref{alg:1}. Henceforth, we use the notation $E_k[\cdot]$ to represent the expectation conditioned on all the randomness up to the $k^{\textit{th}}$ step in Algorithm \ref{alg:1}. The proof appears in Appendix \ref{secAppendix}. 
\begin{pro}
	\label{pronlmid}
	Let $(u^k, \boldsymbol{z}^k)_{k=1}^K$ be the sequence generated by Algorithm \ref{alg:1} with step sizes $\{\gamma_k\}_{k=1}^K$, where $\gamma_k>\mu$ and $\mu$ is  given in Lemma \ref{wcvxt}. 
	For $\rho>0$, let 
	\begin{align}
		\label{nluhzh}
		(\hat{u}^k, \hat{\boldsymbol{z}}^k)=\argmin_{u\in \mathcal{U}, \boldsymbol{z}\in \mathcal{X}}{\varphi(u, \boldsymbol{z})}+\frac{\rho}{2}|\boldsymbol{z}-\boldsymbol{z}^k|^2.
	\end{align}
	Define
	\begin{align}
		\label{nlprxvhgm}
		\varphi_{1/\rho}(\Bar{\boldsymbol{z}})=\min_{u\in \mathcal{U}, \boldsymbol{z}\in \mathcal{X}}\varphi(u, \boldsymbol{z}) + \frac{\rho}{2}|\boldsymbol{z} - \bar{\boldsymbol{z}}|^2. 
	\end{align}
	Then, we get
	\begin{align}
		\label{nllteq}
		&\frac{\rho(\rho-\mu)}{2(\gamma_k - \mu)}|\hat{\boldsymbol{z}}^k - \boldsymbol{z}^k|^2 + \frac{\rho(\gamma_k - \rho)}{2(\gamma_k - \mu)}E_k[|\boldsymbol{z}^{k+1} - \boldsymbol{z}^k|^2]\nonumber\\
		\leq& \varphi_{1/\rho}( \hat{\boldsymbol{z}}^k) - E_k[\varphi_{1/\rho}(\boldsymbol{z}^{k+1})] + \frac{\rho\epsilon_k}{\gamma_k - \mu},
	\end{align}
	where $\epsilon_k$ is given as in \eqref{nonepsdef} with $\gamma:=\gamma_k$. 
\end{pro}

Before we get to the main theorem, the following corollary implies that  $\varphi_{1/\rho}$  in \eqref{nlprxvhgm} is differentiable, which is a direct byproduct of 
Lemma \ref{nlreprelm}.  
\begin{corollary}
	\label{nlproxmp}
	Suppose that Assumption \ref{hypf} holds. Let $\varphi$ be as in \eqref{nlGPvpI}.  Given $\gamma>0$, for $\bar{\boldsymbol{z}}\in \mathcal{X}$, we define
		$(\hat{u}, \hat{\boldsymbol{z}}) = \argmin_{u\in \mathcal{U}, \boldsymbol{z}\in \mathcal{X}}\varphi(u, \boldsymbol{z}) + \frac{\gamma}{2}|\boldsymbol{z} - \bar{\boldsymbol{z}}|^2$. 
	and let $\varphi_{1/\gamma}$ be defined as in \eqref{nlprxvhgm}. 
	Define
	\begin{align}
		\label{nldefpsi}
		\psi({\boldsymbol{z}}) = \frac{\lambda}{2}\boldsymbol{z}^T_{}(\mathcal{K}(\boldsymbol{\phi}_{}, \boldsymbol{\phi}) + \frac{\lambda N}{\beta}I)^{-1}\boldsymbol{z}_{}+ \frac{1}{2N}\sum_{i\in \mathcal{N}}|f_i(z_i) - y_i|^2, \forall \boldsymbol{z}\in \mathcal{X}. 
	\end{align}
	Then,   
	$\hat{\boldsymbol{z}}=\operatorname{prox}_{\psi/\gamma}(\bar{\boldsymbol{z}})$
	and 
		$\varphi_{1/\gamma}(\bar{\boldsymbol{z}}) = \min_{\boldsymbol{z}\in \mathcal{X}}\psi(\boldsymbol{z}) + \frac{\gamma}{2}|\boldsymbol{z} - \bar{\boldsymbol{z}}|^2$. 
	Furthermore, $\varphi_{1/\gamma}$ is differentiable and 
		$\nabla \varphi_{1/\gamma}(\bar{\boldsymbol{z}})  = \gamma(\Bar{\boldsymbol{z}} - \hat{\boldsymbol{z}})$.
\end{corollary}
\begin{proof}
	Let
	\begin{align*}
		(\hat{u}, \hat{\boldsymbol{z}}) = \argmin_{u\in \mathcal{U}, \boldsymbol{z}\in \mathcal{X}}\varphi(u, \boldsymbol{z}) + \frac{\gamma}{2}|\boldsymbol{z} - \bar{\boldsymbol{z}}|^2.
	\end{align*}
	 Thus, by \eqref{nlprxvhgm} and the definition of $\mathcal{L}$ in \eqref{nlLeq},  $
		(\hat{u}, \hat{\boldsymbol{z}}) = \argmin_{u\in \mathcal{U}, \boldsymbol{z}\in \mathcal{X}}\mathcal{L}(u, \boldsymbol{z}; \mathcal{N})$ and  $\varphi_{1/\gamma}(\bar{\boldsymbol{z}}) = \min _{u\in \mathcal{U}, \boldsymbol{z}\in \mathcal{X}}\mathcal{L}(u, \boldsymbol{z}; \mathcal{N})$, where $\mathcal{N}=\{1,\dots, N\}$. 
	Thus, the equation \eqref{nlminLseq} implies that $\hat{\boldsymbol{z}}=\operatorname{prox}_{\psi/\gamma}(\bar{\boldsymbol{z}})$
	and 
	$\varphi_{1/\gamma}(\bar{\boldsymbol{z}}) = \min_{\boldsymbol{z}\in \mathcal{X}}\psi(\boldsymbol{z}) + \frac{\gamma}{2}|\boldsymbol{z} - \bar{\boldsymbol{z}}|^2$.  Hence, $\varphi_{1/\rho}$ is a Moreau envelope. By Lemma \ref{wcvxt},  $\psi$ in \eqref{nldefpsi} is $\mu$-weakly convex in $\boldsymbol{z}$. Then, standard results  \cite{rockafellar1970convex} implies that $\nabla \varphi_{1/\gamma}(\bar{\boldsymbol{z}})  = \gamma(\Bar{\boldsymbol{z}} - \hat{\boldsymbol{z}})$. 
\end{proof}

Next, we have our main theorem about the convergence of Algorithm \ref{alg:1} in the setting of nonlinear measurements. 
\begin{theorem}
	\label{nlmt}
	Suppose that Assumptions \ref{hypnlphibd} and \ref{hypf}  hold. Let $\rho>\mu>0$. In Algorithm \ref{alg:1}, let $\gamma_k = \gamma$ for some $\gamma>\rho$. In each iteration, let $|\mathcal{I}_k|=M$ for all $k\geq 1$, where $M\in \mathbb{N}$. Let $(u^{k}, \boldsymbol{z}^k)_{k=1}^K$ be the sequence generated by Algorithm \ref{alg:1},  let $k^*$ be an index uniformly sampled from $\{1, \dots, K\}$, and let $\varphi_{1/\rho}$ be defined as in \eqref{nlprxvhgm} for $\rho>0$. Then, there exists a constant $C$ depending on $\beta, r, \lambda, L, \mu$ such that
	\begin{align}
		\label{nlmcvt}
		E[|\nabla \varphi_{1/\rho}(\boldsymbol{z}^{k^*})|^2] \leq \frac{2\rho(\gamma - \mu)\Delta}{(\rho - \mu)K} + \frac{C\rho^2}{M(\rho-\mu)\min\{\lambda, \gamma-\mu\}},
	\end{align}
with $\Delta = \varphi_{1/\rho}(\hat{\boldsymbol{z}}^1) - \min_{\boldsymbol{z}\in \mathcal{X}}\varphi_{1/\rho}(\boldsymbol{z})$, where $\hat{\boldsymbol{z}}^1$ is defined as in \eqref{nluhzh}. 
\end{theorem}
\begin{proof}
	Taking $\rho>\mu>0$, we define  $\hat{\boldsymbol{z}}^k$ as  in \eqref{nluhzh}. By Corollary \ref{nlproxmp}, we get $\nabla \varphi_{1/\rho}(\boldsymbol{z}^k) = \rho(\boldsymbol{z}^k - \hat{\boldsymbol{z}}^k)$.  Thus, using Proposition \ref{pronlmid}, we obtain
	\begin{align}
		\label{nlcbp}
		\frac{\rho - \mu}{2\rho(\gamma - \mu)}|\nabla\varphi_{1/\rho}(\boldsymbol{z}^k)|^2 \leq \varphi_{1/\rho}(\hat{\boldsymbol{z}}^k) - E_k[\varphi_{1/\rho}(\boldsymbol{z}^{k+1})] + \frac{\rho\epsilon}{\gamma - \mu},
	\end{align}
	where $\epsilon$ is given in \eqref{nonepsdef}. 
	Summing \eqref{nlcbp} from $k=1$ to $k=K$, taking the expectation over all the randomness, and dividing the result inequality by $K$, we get
	\begin{align}
		\label{nlmcseq}
		\frac{(\rho - \mu)}{2(\gamma - \mu)\rho K}\sum_{k=1}^KE[|\nabla\varphi_{1/\rho}(\boldsymbol{z}^k)|^2] \leq \frac{\varphi_{1/\rho}(\hat{\boldsymbol{z}}^1) - E[\varphi(\boldsymbol{z}^{K+1})]}{K} + \frac{\rho \epsilon}{\gamma - \mu}. 
	\end{align}
	Therefore, using the definition of $k^*$ and Jenssen's inequality, we conclude from \eqref{nlmcseq} and \eqref{nonepsdef} that \eqref{nlmcvt} holds. 
\end{proof}
Theorem \ref{nlmt} shows that $\boldsymbol{z}^{k^*}$ tends close to the neighbor of a nearly stationary point on expectation as the number of iterations and the batch size increase. 
The bound provided in \eqref{nlmcvt} demonstrates that the error does not diminish to zero with an increasing number of iterations for a fixed batch size. Conversely, the accuracy improves as the size of mini-batches increases, a fact supported by the subsequent numerical experiments discussed in the following section.

\begin{figure}[!hbtbp]
	\centering          
	\begin{subfigure}[b]{0.3\textwidth}
		\includegraphics[width=\textwidth]{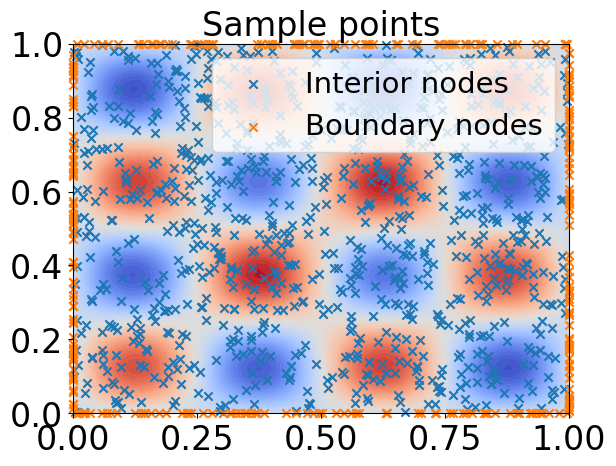}
		\caption{Sample points}
		\label{fig:NonlinearElliptic:sp}
	\end{subfigure} 
	\begin{subfigure}[b]{0.3\textwidth}
		\includegraphics[width=\textwidth]{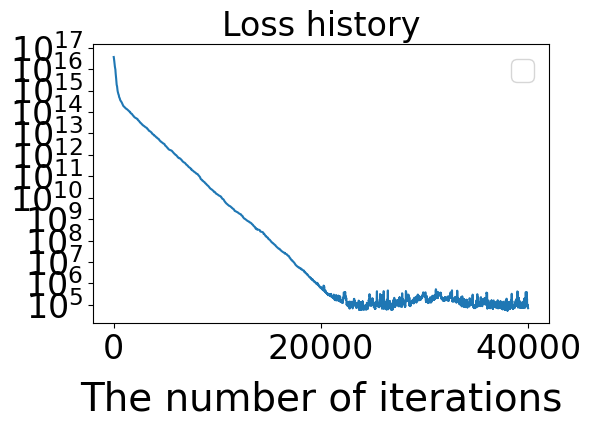}
		\caption{{\color{black}Averaged} Loss history.}
		\label{fig:NonlinearElliptic:lshst1}
	\end{subfigure} 
	\begin{subfigure}[b]{0.3\textwidth}
		\includegraphics[width=\textwidth]{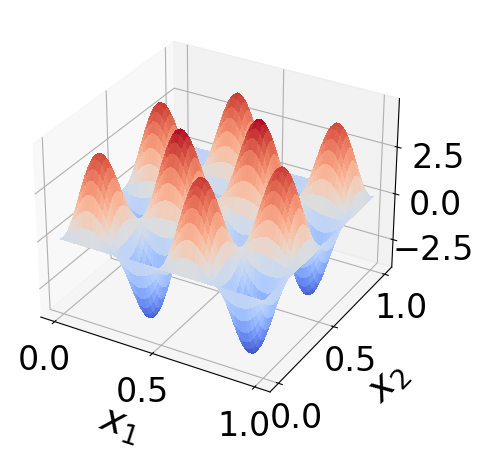}
		\caption{$u^*$.}
		\label{fig:NonlinearElliptic:ustar}
	\end{subfigure} \\
	\begin{subfigure}[b]{0.3\textwidth}
		\includegraphics[width=\textwidth]{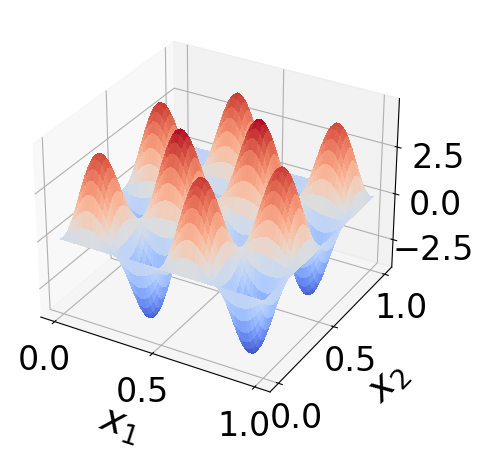}
		\caption{$u_{\text{MGP}}$}
		\label{fig:NonlinearElliptic:gpmu}
	\end{subfigure} 
	\begin{subfigure}[b]{0.3\textwidth}
		\includegraphics[width=\textwidth]{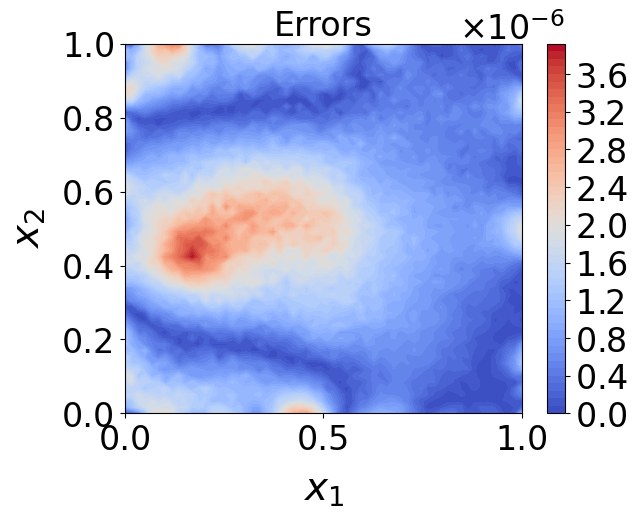}
		\caption{\color{black}Errors of the mini-batch GP method.}
		\label{fig:NonlinearElliptic:error1}
	\end{subfigure} 
\begin{subfigure}[b]{0.3\textwidth}
	\includegraphics[width=\textwidth]{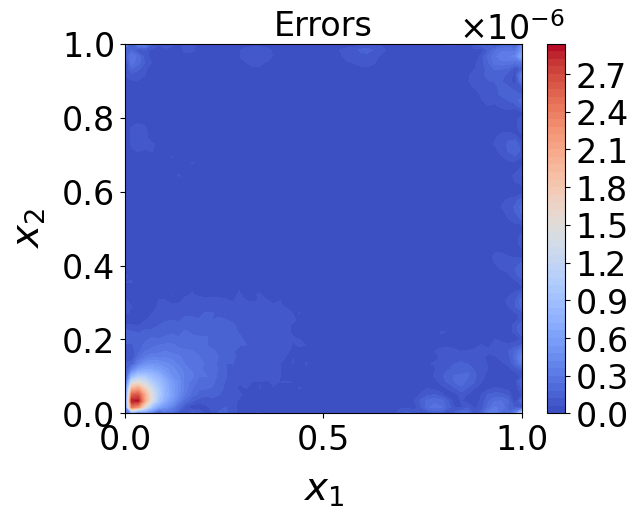}
	\caption{\color{black}Errors of the GP method}
	\label{fig:NonlinearElliptic:errorvanilla}
\end{subfigure} \\
	\caption{Nonlinear elliptic equation: results for the mini-batch GP method using {\color{black}$64$} points in each batch. The parameters $\eta=10^{-13}$. (a): a set of sample points and the contour of the true solution; (b): {\color{black}The convergence history, averaged across the subsequent 10 iterations} (semi-log scale in $y$ axis);  (c): the graph of the true solution $u^*$; (d): the numerical solution $u_{\text{MGP}}$ of the mini-batch GP method; (e): the contour of point-wise errors {\color{black} for our mini-batch GP method}; {\color{black}(f): the contour of point-wise errors of the GP method \cite{chen2021solving}}.}
	\label{fig:NonlinearElliptic:1}
\end{figure}

\section{Numerical Experiments}
\label{secNumResults}
In this section, we demonstrate the efficacy of the mini-batch method by solving a nonlinear elliptic equation in Subsection \ref{subsecNlElliptic} and Burgers' equation in Subsection \ref{subsecBurg}. Our implementation uses Python with the JAX package for automatic differentiation. Our experiments are conducted only on CPUs. To better leverage the information between points in each mini-batch, we take mini-batches with nearby samples instead of sampling uniform mini-batches. A mini-batch consists of a uniformly randomly sampled point and its $M-1$ nearest points. We use $k$-$d$ tree to support such search, which finds $M-1$ nearest for each point in the whole data set of size $N$ in $O(N\log N)$ time and $O(N)$ space. 

\subsection{Solving A Nonlinear Elliptic PDE}
\label{subsecNlElliptic}
\begin{comment}
\begin{figure}[!hbtbp]
	\centering          
	\begin{subfigure}[b]{0.3\textwidth}
		\includegraphics[width=\textwidth]{figures/NonlinearElliptic/2/loss.png}
		\caption{Averaged loss histories. Semi-log scale in $y$ axis.}
		\label{fig:NonlinearElliptic:avgls}
	\end{subfigure} 
	\begin{subfigure}[b]{0.3\textwidth}
		\includegraphics[width=\textwidth]{figures/NonlinearElliptic/2/batch_M_12_M_1200_gpm_error.png}
		\caption{Averaged point-wise errors when $M=12$.}
		\label{fig:NonlinearElliptic:avger12}
	\end{subfigure} \\
	\begin{subfigure}[b]{0.3\textwidth}
		\includegraphics[width=\textwidth]{figures/NonlinearElliptic/2/batch_M_24_M_1200_gpm_error.png}
		\caption{Averaged point-wise errors when $M=24$.}
		\label{fig:NonlinearElliptic:avger24}
	\end{subfigure} 
	\begin{subfigure}[b]{0.3\textwidth}
		\includegraphics[width=\textwidth]{figures/NonlinearElliptic/2/batch_M_48_M_1200_gpm_error.png}
		\caption{Averaged point-wise errors when $M=48$.}
		\label{fig:NonlinearElliptic:avger48}
	\end{subfigure} 
	\caption{Nonlinear elliptic equation: averaged results over $10$ realizations for the mini-batch GP method using different batch sizes. The parameters $\eta=10^{-13}$, $\gamma=1$. (a): averaged loss histories (semi-log scale in $y$ axis); (b): the averaged point-wise errors when $M=12$; (c): the averaged point-wise errors when $M=24$; (e): the averaged point-wise errors when $M=48$. }
	\label{fig:NonlinearElliptic:2}
\end{figure}
\end{comment}

{\color{black}In this example, we solve a nonlinear elliptic equation by utilizing the preconditioned mini-batch method delineated in Remark \ref{rmk:pcd}. The chosen preconditioner, \( Q_k \), is structured as a diagonal matrix. Its diagonal elements are the diagonal entries of the covariance matrix, specific to each set of mini-batch points.} Let $d>1$ and $\Omega=(0,1)^2$. Given a continuous function $f:\Omega \mapsto \mathbb{R}$, we find $u$ solving 
\begin{align}
	\label{illpde}
	\begin{cases}
		-\Delta u(x) + u^3(x) = f(x), x\in \Omega,\\
		u(x) = 0, x\in \partial \Omega. 
	\end{cases}
\end{align}
We prescribe the true solution $u$ to be $\sin(\pi x_1)\sin(\pi x_2)+4\sin(4\pi x_1)\sin(4\pi x_2)$ for $\boldsymbol{x}:=(x_1, x_2)\in \overline{\Omega}$ and compute $f$ at the right-hand side of the first equation of \eqref{illpde} accordingly. For regression, we choose the Gaussian kernel
$K(\boldsymbol{x}, \boldsymbol{y}) =\operatorname{exp}(-\frac{|\boldsymbol{x}-\boldsymbol{y}|^2}{2\sigma^2})$ with the lengthscale parameter $\sigma=0.2$. We take $N=1200$ samples in the domain with $N_\Omega=900$ interior points. 
Meanwhile, we set the nugget  $\eta=10^{-13}$ (see Remark \ref{rmknugget}). 
The algorithm stops after {\color{black}$40000$} iterations. In each time step, we use the technique of eliminating variables at each iteration (see Subsection 3.3 in \cite{chen2021solving}) and solve the resulting nonlinear minimization problem using the Gauss--Newton method, which terminates once the error between two successive steps is less than $10^{-5}$ or the maximum iteration number $30$ is reached.

Figure \ref{fig:NonlinearElliptic:1} shows the results of a realization of the mini-batch GP method when the batch size {\color{black}$M=64$}. Figure \ref{fig:NonlinearElliptic:sp} shows the samples taken in the domain and the contour of the true solution. {\color{black}Figure \ref{fig:NonlinearElliptic:lshst1} depicts the reduction in loss  for the mini-batch GP method during a single run. The loss is derived from the average of the most recent 10 iterations in the algorithm.}. The explicit solution of $u^*$ is given in Figure  \ref{fig:NonlinearElliptic:ustar}. Figure \ref{fig:NonlinearElliptic:gpmu} plots the approximation of $u^*$, denoted by $u_{MGP}$. The contour of the point-wise errors of $u_{MGP}$ and $u^*$ is given in Figure \ref{fig:NonlinearElliptic:error1}. {\color{black}Figure \ref{fig:NonlinearElliptic:errorvanilla} presents the contour  of errors for the GP method as described in \cite{chen2021solving}, using the whole identical sample sets. It is observed that the mini-batch approach with a batch size of 64 achieves accuracy comparable to that of the  GP method using the entire samples. }

\subsection{Burgers' Equation}
\label{subsecBurg}

\begin{figure}[!hbtbp]
	\centering          
	\begin{subfigure}[b]{0.3\textwidth}
		\includegraphics[width=\textwidth]{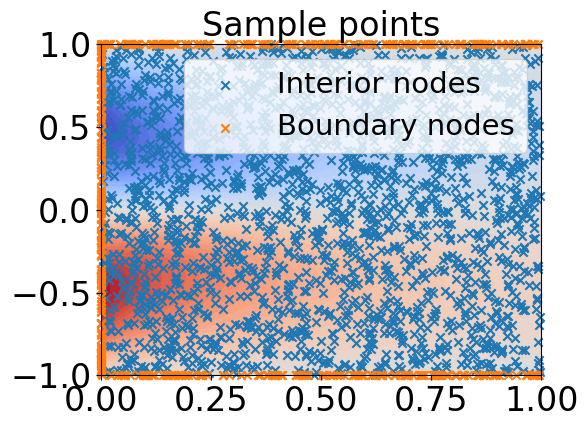}
		\caption{Sample points}
		\label{fig:burgers:sp}
	\end{subfigure} 
	\begin{subfigure}[b]{0.3\textwidth}
		\includegraphics[width=\textwidth]{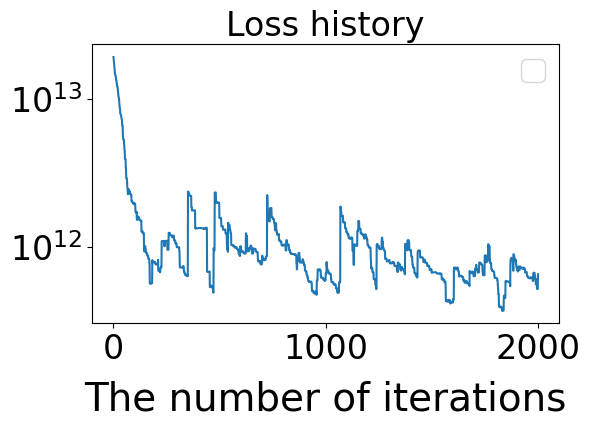}
		\caption{Loss history.}
		\label{fig:burgers:lshst1}
	\end{subfigure} \\
	\begin{subfigure}[b]{0.3\textwidth}
		\includegraphics[width=\textwidth]{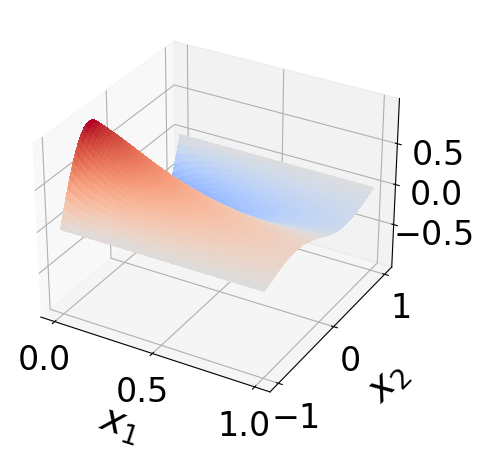}
		\caption{$u^*$.}
		\label{fig:burgers:ustar}
	\end{subfigure} 
	\begin{subfigure}[b]{0.3\textwidth}
		\includegraphics[width=\textwidth]{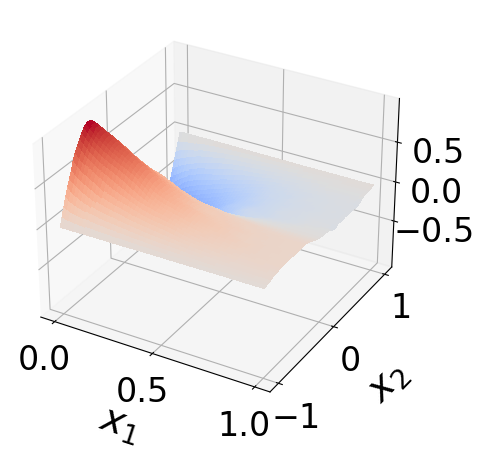}
		\caption{$u_{\text{MGP}}$}
		\label{fig:burgers:gpmu}
	\end{subfigure} 
	\begin{subfigure}[b]{0.325\textwidth}
		\includegraphics[width=\textwidth]{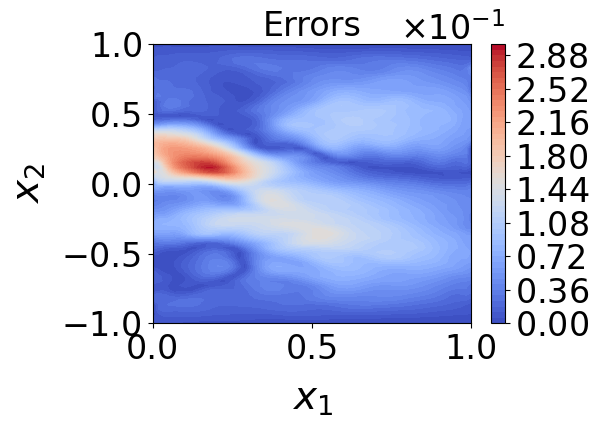}
		\caption{Contour of point-wise errors.}
		\label{fig:burgers:error1}
	\end{subfigure} 
	\caption{Burgers' equation: results for the mini-batch GP method using $75$ points in each batch. The parameters $\eta=10^{-10}$, $\gamma=1$. (a): a set of sample points and the contour of the true solution; (b): the convergence history (semi-log scale in $y$ axis); (c): the graph of the true solution $u^*$; (d): the numerical solution $u_{\text{MGP}}$ of the mini-batch GP method; (e): the contour of point-wise errors. }
	\label{fig:Burgers:1}
\end{figure}

In this subsection, we consider Burgers' equation
\begin{align*}
	\begin{cases}
		\partial_t u + u \partial_x u + \nu \partial_x^2 u = 0, \forall (t, x)\in (0, 1]\times (-1, 1),\\
		u(0, x) = -\sin(\pi x), \forall x\in [-1, 1],\\
		u(t, -1) = u(t, 1)=0, \forall t\in [0, 1],
	\end{cases}
\end{align*}
where $\nu=0.2$ and  $u$ is unknown. 
We take $N=2400$ points uniformly in the space-time domain, of which $N_\Omega=2000$ points are in the domain's interior. As in \cite{chen2021solving}, we choose the anisotropic kernel
\begin{align}
	\label{anisotropickn}
	K((t, x), (t', x'))=\operatorname{exp}(-(t-t')^2/\sigma_1^2-(x-x')^2/{\sigma_2^2}),
\end{align}
with $(\sigma_1, \sigma_2)=(0.3, 0.05)$. To compute pointwise errors, as in \cite{chen2021solving}, we compute the true solution using the Cole--Hopf transformation and the numerical quadrature. We use the technique of eliminating constraints at each iteration (see Subsection 3.3 in \cite{chen2021solving}) and use the Gauss--Newton method to solve the resulting nonlinear quadratic minimization problem, which terminates once the error between two successive steps is less than $10^{-5}$ or the maximum iteration number $100$ is reached. Throughout the experiments, we use the nugget $\eta=10^{-10}$ (see Remark \ref{rmknugget}). 

\begin{comment}
{\color{black}The convergence of our mini-batch methodology depends on the initial parameters set during the algorithm's configuration. A continuation strategy is employed to refine these starting conditions. Initially, the Burgers' equation with a viscosity coefficient of \(\nu=1\) is solved using the mini-batch approach. Subsequent iterations solve the equation with reduced values of \(\nu\), each time initializing the latent variable \(\boldsymbol{z}\) with the solution from the preceding iteration, corresponding to a higher viscosity.
	
	Figure 3 displays the results obtained when the latent variable \(\boldsymbol{z}\) is initialized with the solution for Burgers' equation at \(\nu=0.03\). Figure 3a illustrates the loss function's trajectory over the iterations. Figure 3b compares pointwise errors, juxtaposing the mini-batch method's outcomes with the reference solution. Conversely, Panel 3c shows the pointwise discrepancies from the GP method, without mini-batching, against the reference solution. This analysis indicates that the mini-batch approach achieves a precision comparable to the GP method implemented without mini-batching.
}
\end{comment}

\begin{figure}[!hbtbp]
	\centering          
	\begin{subfigure}[b]{0.3\textwidth}
		\includegraphics[width=\textwidth]{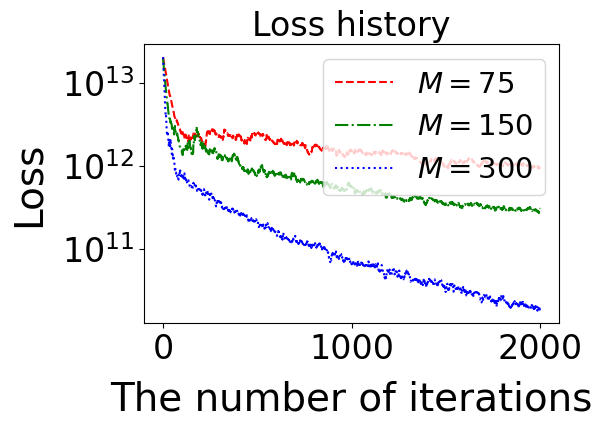}
		\caption{Averaged loss histories. Semi-log scale in $y$ axis.}
		\label{fig:burgers:avgls}
	\end{subfigure} 
	\begin{subfigure}[b]{0.3\textwidth}
		\includegraphics[width=\textwidth]{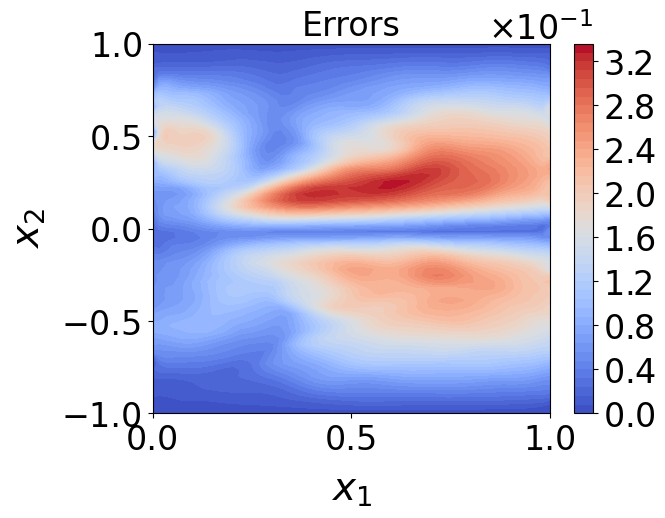}
		\caption{Averaged point-wise errors when $M=75$.}
		\label{fig:burgers:avger12}
	\end{subfigure} \\
	\begin{subfigure}[b]{0.325\textwidth}
		\includegraphics[width=\textwidth]{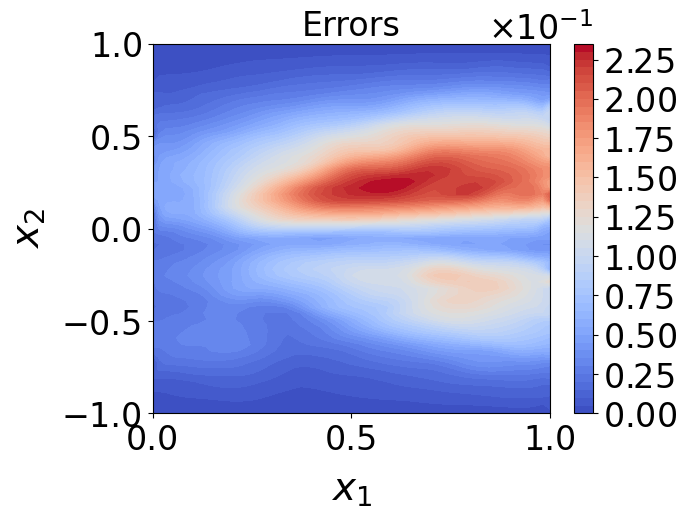}
		\caption{Averaged point-wise errors when $M=150$.}
		\label{fig:burgers:avger24}
	\end{subfigure} 
	\begin{subfigure}[b]{0.3\textwidth}
		\includegraphics[width=\textwidth]{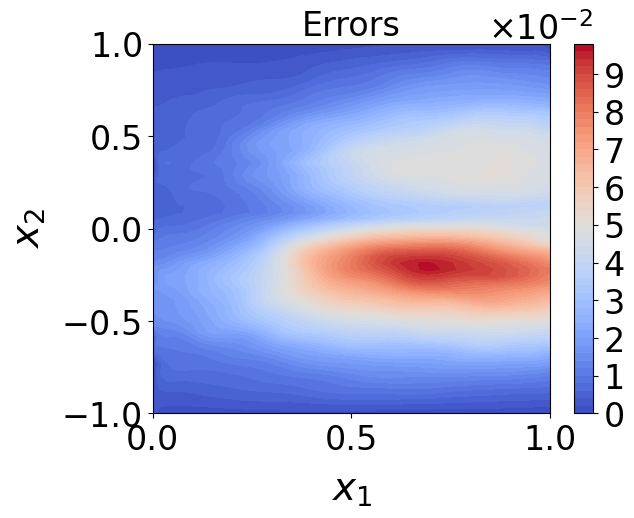}
		\caption{Averaged point-wise errors when $M=300$.}
		\label{fig:burgers:avger48}
	\end{subfigure} 
	\caption{Burgers' equation: averaged results over $10$ realizations for the mini-batch GP method using different batch sizes. The parameters $\eta=10^{-10}$, $\gamma=1$. (a): averaged loss histories (semi-log scale in $y$ axis); (b): the averaged point-wise errors when $M=75$; (c): the averaged point-wise errors when $M=150$; (e): the averaged point-wise errors when $M=300$. }
	\label{fig:burgers:2}
\end{figure}

Figure \ref{fig:Burgers:1} shows the results of a realization of the mini-batch GP method when we fix the batch size $M=75$. Figure \ref{fig:burgers:sp} shows the samples in the domain and the contour of the true solution. Figure \ref{fig:burgers:lshst1} illustrates the history of the loss generated by the min-batch GP method in one realization. The explicit solution of $u^*$ is plotted in Figure  \ref{fig:burgers:ustar}. Figure \ref{fig:burgers:gpmu} shows the numerical approximation of $u^*$, denoted by $u_{MGP}$. The contour of the point-wise errors between $u_{MGP}$ and $u^*$ is presented in Figure \ref{fig:burgers:error1}.

Then, we run the algorithm $10$ times for several batch sizes and plot the results in Figure \ref{fig:burgers:2}. The averaged loss histories are plotted in Figure \ref{fig:burgers:avgls}, from which we see that  increasing batch sizes results in faster convergence rates and smaller losses. The contours of averaged point-wise errors are presented in Figures \ref{fig:burgers:avger12}-\ref{fig:burgers:avger48}, which imply that a larger batch size achieves better accuracy. 

In our observations regarding the solution of Burgers' equation, it becomes evident that a larger number of samples in a mini-batch is required compared to the nonlinear elliptic example. This disparity could be attributed to the relatively lower regularity of Burgers' equation, where the uniformly distributed samples may not adequately capture the inherent regularity. Consequently, we defer the exploration of more effective strategies for sampling collocation points and selecting mini-batch points for future investigations.

\section{Conclusions}
\label{secConclu}
This paper presents a mini-batch framework to solve nonlinear PDEs with GPs.
The mini-batch method allocates the computational cost of GPs to each iterative step. We also perform a rigorous convergence analysis of the mini-batch method. For future work, extending our framework to more general GP regression problems like semi-supervised learning and hyperparameter learning would be interesting. In the numerical experiments, we notice that 
the choice of the positions of mini-batch samples influences the performance. Hence, a potential future work is investigating different sampling techniques to take mini-batch sample points. 

{
\color{black}
\begin{acknowledgement*}
The authors wish to thank Professor Andrew Stuart for his insightful discussions on the convergence of the stochastic gradient descent method.
\end{acknowledgement*}
}

\appendix

\section{Proofs of Results}
\label{secAppendix}
\begin{proof}[Proof of Lemma \ref{nlreprelm}]
	Let 
	$\mathcal{L}(u, \boldsymbol{z}; \mathcal{I}) = \varphi(u, \boldsymbol{z}; \mathcal{I}) + \frac{\gamma}{2}|\boldsymbol{z} - \Bar{\boldsymbol{z}}|^2, \forall u\in \mathcal{U}, \boldsymbol{z}\in \mathcal{X}$. 
	Define the sampling operator  ${S}_{\mathcal{I}}$ as in Lemma \eqref{solm}. Then, we have
	\begin{align}
		\label{nlLeq}
		\mathcal{L}(u, \boldsymbol{z}; \mathcal{I}) 
		=& \frac{1}{2}\bigg\langle {u}, \lambda {u} + \frac{\beta}{M}S_{\mathcal{I}}^*S_{\mathcal{I}}{u} - \frac{2\beta}{M}S_{\mathcal{I}}^*\boldsymbol{z}_{\mathcal{I}}\bigg\rangle_{\mathcal{U}}\nonumber \\
		&+ \frac{\beta}{2M}|\boldsymbol{z}_{\mathcal{I}}|^2 + \frac{1}{2M}\sum_{i\in \mathcal{I}}|f_i(z_i)-y_i|^2 + \frac{\gamma}{2}|\boldsymbol{z} - \Bar{\boldsymbol{z}}|^2.
	\end{align}
	Next, we first fix $\boldsymbol{z}\in \mathcal{X}$ and let  $\tilde{u}=\argmin_{u}\mathcal{L}(u, \boldsymbol{z};\mathcal{I})$.  
	Then, 
	$$\frac{\partial\mathcal{L}}{\partial u}\bigg |_{u=\Tilde{u}} = \lambda \Tilde{u} + \frac{\beta}{M}S_{\mathcal{I}}^*S_{\mathcal{I}}\Tilde{u} - \frac{\beta}{M}S_{\mathcal{I}}^*\boldsymbol{z}_{\mathcal{I}} = 0, $$
	which yields
	\begin{align}
		\label{nludfml}
		\Tilde{u} = \bigg(\frac{\lambda M}{\beta}I + S^*_{\mathcal{I}}S_{\mathcal{I}} \bigg)^{-1}S_{\mathcal{I}}^*\boldsymbol{z}_{\mathcal{I}}. 
	\end{align}
	Thus, using \eqref{nlLeq} and \eqref{nludfml}, we have
	\begin{align}
		\min_{u\in \mathcal{U}}\mathcal{L}(u, \boldsymbol{z}; \mathcal{I}) 
		=& -\frac{\beta}{2M}\langle S_{\mathcal{I}}\Tilde{u}, \boldsymbol{z}_{\mathcal{I}}\rangle + \frac{\beta}{2M}|\boldsymbol{z}_{\mathcal{I}}|^2 + \frac{1}{2M}\sum_{i\in \mathcal{I}}|f_i(z_i) - y_i|^2 + \frac{\gamma}{2}|\boldsymbol{z} - \Bar{\boldsymbol{z}}|^2\nonumber\\
		=& \frac{\lambda}{2}\boldsymbol{z}^T_{\mathcal{I}}\bigg(\mathcal{K}(\boldsymbol{\phi}_{\mathcal{I}}, \boldsymbol{\phi}_{\mathcal{I}}) + \frac{\lambda M}{\beta}I\bigg)^{-1}\boldsymbol{z}_{\mathcal{I}}+ \frac{1}{2M}\sum_{i\in \mathcal{I}}|f_i(z_i) - y_i|^2 +  \frac{\gamma}{2}|\boldsymbol{z} - \Bar{\boldsymbol{z}}|^2. \label{nlminLseq}
	\end{align} 
	where the last equality results from \eqref{bdImus}. 
	Therefore, \eqref{nludfml},  \eqref{eqre0}, and  \eqref{nlminLseq} imply that the minimizer $(u^\dagger, \boldsymbol{z}^\dagger)$ in \eqref{nlhproxpr} satisfies \eqref{nluzprli} and $\boldsymbol{z}^\dagger$ minimizes \eqref{nlmzp}.
\end{proof}

\begin{proof}[Proof of Lemma \ref{nlstles}]
	Let ${\mathcal{I}}' $ and ${\mathcal{I}}_{(i)}$ be  as in the beginning of this subsection. For $\bar{\boldsymbol{z}}\in \mathcal{X}$, let 
	\begin{align}
		(\hat{u}, \hat{\boldsymbol{z}}) = \widetilde{\operatorname{prox}}_{\varphi/\gamma}^{\mathcal{I}}(\bar{\boldsymbol{z}}) \text{ and }
		(\hat{u}_{(i)}, \hat{\boldsymbol{z}}_{(i)}) = \widetilde{\operatorname{prox}}_{\varphi/\gamma}^{{\mathcal{I}}_{(i)}}(\bar{\boldsymbol{z}}). \label{stbeqnll}
	\end{align}
	For $\gamma>1$, by the optimality condition of \eqref{stbeqnll}, there exists a constant $C$ depending on $\beta$, $\lambda$,  the diameter $r$ of $\mathcal{X}$, and the constant $U_f$ in Assumption \ref{hypf}, such that 
	\begin{align}
		\label{bduhnl}
		\|\hat{u}\|_{\mathcal{U}}^2\leq \frac{2}{\lambda}\varphi(0, \bar{\boldsymbol{z}};{\mathcal{I}})\leq \frac{\beta}{M\lambda}|\bar{\boldsymbol{z}}|^2+\frac{1}{\lambda M}\sum_{i\in \mathcal{I}}|f_i(\bar{\boldsymbol{z}}_i)-y_i|^2\leq C.
	\end{align}
	Thus, using \eqref{bduhnl} and Assumption \ref{hypnlphibd}, there exists a constant $C$ depending on $\beta$, $\lambda$, $r$, $U_f$, and the constant $L$ in Assumption \ref{hypnlphibd}, such that 
	$|[\boldsymbol{\phi}_i, \hat{u}]|\leq \|\boldsymbol{\phi}_i\|_{\mathcal{U}^*}\|\hat{u}\|_{\mathcal{U}}\leq C$. 
	Let $\mathcal{B}=\{u\in \mathcal{U}|\|u\|_{\mathcal{U}} \leq \Tilde{C}\}\times \mathcal{X}$, where $\Tilde{C}$ is large enough  such that $\mathcal{B}$ contains $(\hat{u}, \hat{\boldsymbol{z}})$ in \eqref{stbeqnll} for any $\bar{\boldsymbol{z}}\in \mathcal{X}$. 
	Next, we show that for any $\xi\in \{1, \dots, N\}$, $\varphi(\cdot, \cdot;\xi)$ is Lipschitz in $\mathcal{B}$.  Let $(u_1, \boldsymbol{z}_1), (u_2, \boldsymbol{z}_2) \in \mathcal{B}$. Then,  there exists a constant $C$ depending on $\beta, r, \lambda, L$, and ${U}_f$,  such that 
	\begin{align}
		&|\varphi(u_1, \boldsymbol{z}_1; \xi) - \varphi(u_2, \boldsymbol{z}_2; \xi)|\nonumber\\
		\leq&  \frac{\lambda}{2}\|u_1 + u_2\|_{\mathcal{U}}\|u_1 - u_2\|_{\mathcal{U}}+ \frac{1}{2}\big|f_{\xi}(\boldsymbol{z}_{1,\xi})+f_{\xi}(\boldsymbol{z}_{2,\xi})-2y_{\xi}\big| |f_{\xi}(\boldsymbol{z}_{1,\xi}) - f_{\xi}(\boldsymbol{z}_{2, \xi})|\nonumber 
		\\
		&+ 
		\frac{\beta}{2}\bigg|[\boldsymbol{\phi}_{\xi}, u_1] - \boldsymbol{z}_{1, \xi}+[\boldsymbol{\phi}_{\xi}, u_2]-\boldsymbol{z}_{2,\xi} \bigg|\bigg(\big|[\boldsymbol{\phi}_{\xi}, u_1]-[\boldsymbol{\phi}_{\xi}, u_2] \big| + \big| \boldsymbol{z}_{1,\xi}-\boldsymbol{z}_{2,\xi} \big| \bigg)\nonumber\\
		\leq & C(\|u_1 - u_2\|_{\mathcal{U}} + |\boldsymbol{z}_1 - \boldsymbol{z}_2|), \label{vplinnls}
	\end{align} 
	where \eqref{vplinnls} follows by the definition of $\mathcal{B}$ and Assumptions \ref{hypnlphibd} and \ref{hypf}. 
	
	By Jensen's inequality, we obtain 
	\begin{align}
		\label{stbcmp}
		|E_{\mathcal{I}, \mathcal{I}', i}[\varphi(\hat{u}, \hat{\boldsymbol{z}};\xi_i') - \varphi(\hat{u}_{(i)}, \hat{\boldsymbol{z}}_{(i)}; \xi_i')]|
		\leq & \frac{1}{M}\sum_{i=1}^ME_{\mathcal{I}, \xi_i'}\bigg|\varphi(\hat{u}, \hat{\boldsymbol{z}};\xi_i')-\varphi(\hat{u}_{(i)}, \hat{\boldsymbol{z}}_{(i)}; \xi_i')\bigg|\nonumber\\
		\leq&  \frac{C}{M}\sum_{i=1}^ME_{\mathcal{I}, \xi_i'}(\|\hat{u}-\hat{u}_{(i)}\|_{\mathcal{U}} + |\hat{\boldsymbol{z}}-\hat{\boldsymbol{z}}_{(i)}| ),
	\end{align}
	where \eqref{stbcmp} follows by \eqref{vplinnls}. Using \eqref{stbeqnll}, Corollary \ref{corovarphicv}, and Lemma \ref{wc3plemma}, we get
	\begin{align}
		&\varphi(\hat{u}, \hat{\boldsymbol{z}}; {\mathcal{I}}) + \frac{\gamma}{2}|\hat{\boldsymbol{z}}-\bar{\boldsymbol{z}}|^2\nonumber\\
		\leq& \varphi(\hat{u}_{(i)}, \hat{\boldsymbol{z}}_{(i)}; {\mathcal{I}}) + \frac{\gamma}{2}|\hat{\boldsymbol{z}}_{(i)}-\bar{\boldsymbol{z}}|^2 - \frac{\gamma-\mu}{2}|\hat{\boldsymbol{z}}-\hat{\boldsymbol{z}}_{(i)}|^2-\frac{\lambda}{2}\|\hat{u}-\hat{u}_{(i)}\|_{\mathcal{U}}^2, \label{vpnluhuhi:eq0}\\
		&\varphi(\hat{u}_{(i)}, \hat{\boldsymbol{z}}_{(i)}; {\mathcal{I}}_{(i)}) + \frac{\gamma}{2}|\hat{\boldsymbol{z}}_{(i)}-\bar{\boldsymbol{z}}|^2\nonumber\\
		\leq& \varphi(\hat{u}, \hat{\boldsymbol{z}}; {\mathcal{I}}_{(i)}) + \frac{\gamma}{2}|\hat{\boldsymbol{z}}-\bar{\boldsymbol{z}}|^2-\frac{\gamma-\mu}{2}|\hat{\boldsymbol{z}}-\hat{\boldsymbol{z}}_{(i)}|^2-\frac{\lambda}{2}\|\hat{u}-\hat{u}_{(i)}\|_{\mathcal{U}}^2. \label{vpnluhuhi:eq1}
	\end{align}
	Adding \eqref{vpnluhuhi:eq0} and \eqref{vpnluhuhi:eq1} together and using \eqref{vplinnls}, we obtain 
	\begin{align}
		&\lambda\|\hat{u}-\hat{u}_{(i)}\|^2_{\mathcal{U}}+(\gamma-\mu)|\hat{\boldsymbol{z}}-\hat{\boldsymbol{z}}_{(i)}|^2\nonumber\\
		\leq& \varphi(\hat{u}_{(i)}, \hat{\boldsymbol{z}}_{(i)}; {\mathcal{I}}) - \varphi(\hat{u}_{(i)}, \hat{\boldsymbol{z}}_{(i)}; {\mathcal{I}}_{(i)}) + \varphi(\hat{u}, \hat{\boldsymbol{z}}; {\mathcal{I}}_{(i)}) - \varphi(\hat{u}, \hat{\boldsymbol{z}}; {\mathcal{I}})\nonumber\\
		= & \frac{1}{M}[\varphi(\hat{u}_{(i)}, \hat{\boldsymbol{z}}_{(i)}; \xi_i) - \varphi(\hat{u}_{(i)}, \hat{\boldsymbol{z}}_{(i)}; \xi_i') + \varphi(\hat{u}, \hat{\boldsymbol{z}}; \xi_i') - \varphi(\hat{u}, \hat{\boldsymbol{z}}; \xi_i)]\nonumber\\
		\leq & \frac{C}{M}(\|\hat{u}- \hat{u}_{(i)}\|_{\mathcal{U}} + |\hat{\boldsymbol{z}}-\hat{\boldsymbol{z}}_{(i)}|),\nonumber
	\end{align}
	which together with the fact that $\forall a, b \in \mathbb{R}, (a + b)^2 \leq 2a^2 + 2b^2$,  yield
	\begin{align}
		\label{bduuizzinl}
		\|\hat{u}-\hat{u}_{(i)}\|_{\mathcal{U}}+|\hat{\boldsymbol{z}}-\hat{\boldsymbol{z}}_{(i)}| \leq \frac{C}{M\min\{\lambda, \gamma-\mu\}}.
	\end{align}
	Hence,  \eqref{stbcmp} and \eqref{bduuizzinl} imply that $\widetilde{\operatorname{prox}}_{\varphi/\gamma}^{\mathcal{I}}$ is $\epsilon$-stable with $\epsilon$ given in \eqref{nonepsdef}. 
\end{proof}

\begin{proof}[Proof of Proposition \ref{pronlmid}]
	Let $\mathcal{I}_k$ be a mini-batch of indices in the $k^{\textit{th}}$ step of Algorithm \ref{alg:1}. 
	According to Corollary \ref{corovarphicv}, $(u, \boldsymbol{z})\mapsto\varphi(u, \boldsymbol{z}; \mathcal{I}_k)$ is $\mu$-weakly convex in $\boldsymbol{z}$ and strongly convex in $u$. Then, by Lemma \ref{wc3plemma} and \eqref{alg:1:upts}, for any $u\in \mathcal{U}$ and $\boldsymbol{z}\in \mathcal{X}$, we have  
	\begin{align}
		\frac{\gamma_k - \mu}{2}|\boldsymbol{z} - \boldsymbol{z}^{k+1}|^2 \leq 
	& \frac{\gamma_k}{2}|\boldsymbol{z} - \boldsymbol{z}^{k}|^2 - \frac{\gamma_k}{2}|\boldsymbol{z}^{k+1} - \boldsymbol{z}^k|^2 +  \varphi(u, \boldsymbol{z}; {\mathcal{I}}_k) - \varphi(u^{k+1}, \boldsymbol{z}^{k+1})\nonumber \\
		&+ \varphi(u^{k+1}, \boldsymbol{z}^{k+1}) -  \varphi(u^{k+1}, \boldsymbol{z}^{k+1}; {\mathcal{I}}_k). \label{nlzine}
	\end{align}
	Lemmas \ref{mmberdb} and \ref{nlstles} imply that
	\begin{align}
		\label{stb}
		E_k[\varphi(u^{k+1}, \boldsymbol{z}^{k+1}) -  \varphi(u^{k+1}, \boldsymbol{z}^{k+1}; {\mathcal{I}}_k)] \leq \epsilon_k, 
	\end{align}
	where $\epsilon_k$ is given in \eqref{nonepsdef} with $\gamma:=\gamma_k$. 
	For $\rho>0$,  let 	$(\hat{u}^k, \hat{\boldsymbol{z}}^k)$ be given in \eqref{nluhzh}. 
	Then, we have
	\begin{align}
		\label{vpuhzhnlmuk1}
		\varphi(\hat{u}^k, \hat{\boldsymbol{z}}^k) - \varphi(u^{k+1}, \boldsymbol{z}^{k+1}) \leq - \frac{\rho}{2}|\hat{\boldsymbol{z}}^k - \boldsymbol{z}^k|^2 + \frac{\rho}{2}|\boldsymbol{z}^{k+1} - \boldsymbol{z}^k|^2. 
	\end{align}
	Plugging $(u, \boldsymbol{z}):=(\hat{u}^k, \hat{\boldsymbol{z}}^k)$ into \eqref{nlzine}, taking the expectation, using \eqref{stb} and \eqref{vpuhzhnlmuk1}, and using Lemma \ref{nlstles}, we obtain
	\begin{align}
		\label{nldccs}
		\frac{\gamma_k-\mu}{2} E_k[|\hat{\boldsymbol{z}}^k - \boldsymbol{z}^{k+1}|^2] \leq \frac{\gamma_k-\rho}{2}|\hat{\boldsymbol{z}}^k - \boldsymbol{z}^k|^2 - \frac{\gamma_k-\rho}{2}E_k[|\boldsymbol{z}^{k+1} - \boldsymbol{z}^k|^2] + \epsilon_k. 
	\end{align}
	For $\rho>0$, using the definition of $\varphi_{1/\rho}$ and \eqref{nldccs}, we have
	\begin{align}
		&E_k[\varphi_{1/\rho}(\boldsymbol{z}^{k+1})] 
		\leq  E_k[\varphi(\hat{u}^k, \hat{\boldsymbol{z}}^k) + \frac{\rho}{2}|\hat{\boldsymbol{z}}^k - \boldsymbol{z}^{k+1}|^2]\nonumber\\
		\leq & \varphi_{1/\rho}(\hat{\boldsymbol{z}}^k) - \frac{\rho(\rho-\mu)}{2(\gamma_k - \mu)}|\hat{\boldsymbol{z}}^k - \boldsymbol{z}^k|^2 - \frac{\rho(\gamma_k - \rho)}{2(\gamma_k - \mu)}E_k[|\boldsymbol{z}^{k+1} - \boldsymbol{z}^k|^2] + \frac{\rho\epsilon_k}{\gamma_k - \mu},\nonumber
	\end{align}
	which concludes \eqref{nllteq} after rearrangement. 
\end{proof}


\begin{thebibliography}{10}
	
	\bibitem{asi2019importance}
	{\sc H.~Asi and J.~C. Duchi}, {\em The importance of better models in
		stochastic optimization}, Proceedings of the National Academy of Sciences,
	116 (2019), pp.~22924--22930.
	
	\bibitem{asi2019stochastic}
	{\sc H.~Asi and J.~C. Duchi}, {\em Stochastic (approximate) proximal point
		methods: Convergence, optimality, and adaptivity}, SIAM Journal on
	Optimization, 29 (2019), pp.~2257--2290.
	
	\bibitem{batlle2023error}
	{\sc P.~Batlle, Y.~Chen, B.~Hosseini, H.~Owhadi, and A.~M. Stuart}, {\em Error
		analysis of kernel/{GP} methods for nonlinear and parametric {PDE}s}, arXiv
	preprint arXiv:2305.04962,  (2023).
	
	\bibitem{bertsekas2011incremental}
	{\sc D.~P. Bertsekas}, {\em Incremental proximal methods for large scale convex
		optimization}, Mathematical programming, 129 (2011), pp.~163--195.
	
	\bibitem{bottou2007tradeoffs}
	{\sc L.~Bottou and O.~Bousquet}, {\em The tradeoffs of large scale learning},
	Advances in neural information processing systems, 20 (2007).
	
	\bibitem{bousquet2002stability}
	{\sc O.~Bousquet and A.~Elisseeff}, {\em Stability and generalization}, The
	Journal of Machine Learning Research, 2 (2002), pp.~499--526.
	
	\bibitem{chen2020stochastic}
	{\sc H.~Chen, L.~Zheng, R.~Al~Kontar, and G.~Raskutti}, {\em Stochastic
		gradient descent in correlated settings: A study on {G}aussian processes},
	Advances in neural information processing systems, 33 (2020), pp.~2722--2733.
	
	\bibitem{chen2022gaussian}
	{\sc H.~Chen, L.~Zheng, R.~Al~Kontar, and G.~Raskutti}, {\em Gaussian process
		parameter estimation using mini-batch stochastic gradient descent:
		convergence guarantees and empirical benefits}, The Journal of Machine
	Learning Research, 23 (2022), pp.~10298--10356.
	
	\bibitem{chen2021solving}
	{\sc Y.~Chen, B.~Hosseini, H.~Owhadi, and A.~M. Stuart}, {\em Solving and
		learning nonlinear {PDE}s with {G}aussian processes}, Journal of
	Computational Physics,  (2021).
	
	\bibitem{chen2023sparse}
	{\sc Y.~Chen, H.~Owhadi, and F.~Sch{\"a}fer}, {\em Sparse {C}holesky
		factorization for solving nonlinear {PDE}s via {G}aussian processes}, arXiv
	preprint arXiv:2304.01294,  (2023).
	
	\bibitem{damianou2016variational}
	{\sc A.~C. Damianou, M.~K. Titsias, and N.~D. Lawrence}, {\em Variational
		inference for latent variables and uncertain inputs in {G}aussian processes},
	(2016).
	
	\bibitem{davis2019stochastic}
	{\sc D.~Davis and D.~Drusvyatskiy}, {\em Stochastic model-based minimization of
		weakly convex functions}, SIAM Journal on Optimization, 29 (2019),
	pp.~207--239.
	
	\bibitem{deng2021minibatch}
	{\sc Q.~Deng and W.~Gao}, {\em Minibatch and momentum model-based methods for
		stochastic weakly convex optimization}, Advances in Neural Information
	Processing Systems, 34 (2021), pp.~23115--23127.
	
	\bibitem{drusvyatskiy2019efficiency}
	{\sc D.~Drusvyatskiy and C.~Paquette}, {\em Efficiency of minimizing
		compositions of convex functions and smooth maps}, Mathematical Programming,
	178 (2019), pp.~503--558.
	
	\bibitem{duchi2018stochastic}
	{\sc J.~C. Duchi and F.~Ruan}, {\em Stochastic methods for composite and weakly
		convex optimization problems}, SIAM Journal on Optimization, 28 (2018),
	pp.~3229--3259.
	
	\bibitem{gardner2018gpytorch}
	{\sc J.~Gardner, G.~Pleiss, K.~Q. Weinberger, D.~Bindel, and A.~G. Wilson},
	{\em G{P}ytorch: Blackbox matrix-matrix {G}aussian process inference with
		{GPU} acceleration}, Advances in neural information processing systems, 31
	(2018).
	
	\bibitem{hensman2017variational}
	{\sc J.~Hensman, N.~Durrande, and A.~Solin}, {\em Variational {F}ourier
		features for {G}aussian processes.}, J. Mach. Learn. Res., 18 (2017),
	pp.~5537--5588.
	
	\bibitem{Hensman_2013}
	{\sc J.~Hensman, N.~Fusi, and N.~D. Lawrence}, {\em {G}aussian processes for
		big data}, in Proceedings of the Twenty-Ninth Conference on Uncertainty in
	Artificial Intelligence, UAI'13, Arlington, Virginia, United States, 2013,
	AUAI Press, pp.~282--290,
	\url{http://dl.acm.org/citation.cfm?id=3023638.3023667}.
	
	\bibitem{hughes2012finite}
	{\sc T.~J. Hughes}, {\em The finite element method: linear static and dynamic
		finite element analysis}, Courier Corporation, 2012.
	
	\bibitem{karniadakis2021physics}
	{\sc G.~E. Karniadakis, I.~G. Kevrekidis, L.~Lu, P.~Perdikaris, S.~Wang, and
		L.~Yang}, {\em Physics-informed machine learning}, Nature Reviews Physics, 3
	(2021), pp.~422--440.
	
	\bibitem{kramer2021probabilistic}
	{\sc N.~Kr{\"a}mer, J.~Schmidt, and P.~Hennig}, {\em Probabilistic numerical
		method of lines for time-dependent partial differential equations}, arXiv
	preprint arXiv:2110.11847,  (2021).
	
	\bibitem{kulis2010implicit}
	{\sc B.~Kulis and P.~L. Bartlett}, {\em Implicit online learning}, in
	Proceedings of the 27th International Conference on Machine Learning
	(ICML-10), 2010, pp.~575--582.
	
	\bibitem{lazaro2009inter}
	{\sc M.~L{\'a}zaro-Gredilla and A.~Figueiras-Vidal}, {\em Inter-domain
		{G}aussian processes for sparse inference using inducing features}, Advances
	in Neural Information Processing Systems, 22 (2009).
	
	\bibitem{lazaro2010sparse}
	{\sc M.~L{\'a}zaro-Gredilla, J.~Quinonero-Candela, C.~E. Rasmussen, and A.~R.
		Figueiras-Vidal}, {\em Sparse spectrum {G}aussian process regression}, The
	Journal of Machine Learning Research, 11 (2010), pp.~1865--1881.
	
	\bibitem{li2017hyperband}
	{\sc L.~Li, K.~Jamieson, G.~DeSalvo, A.~Rostamizadeh, and A.~Talwalkar}, {\em
		Hyperband: A novel bandit-based approach to hyperparameter optimization}, The
	Journal of Machine Learning Research, 18 (2017), pp.~6765--6816.
	
	\bibitem{liu2020gaussian}
	{\sc H.~Liu, Y.~S. Ong, X.~Shen, and J.~Cai}, {\em When {G}aussian process
		meets big data: A review of scalable {GP}s}, IEEE transactions on neural
	networks and learning systems, 31 (2020), pp.~4405--4423.
	
	\bibitem{lu2021learning}
	{\sc L.~Lu, P.~Jin, G.~Pang, Z.~Zhang, and G.~E. Karniadakis}, {\em Learning
		nonlinear operators via deeponet based on the universal approximation theorem
		of operators}, Nature Machine Intelligence, 3 (2021), pp.~218--229.
	
	\bibitem{mao2020physics}
	{\sc Z.~Mao, A.~D. Jagtap, and G.~E. Karniadakis}, {\em Physics-informed neural
		networks for high-speed flows}, Computer Methods in Applied Mechanics and
	Engineering, 360 (2020), p.~112789.
	
	\bibitem{meng2023sparse}
	{\sc R.~Meng and X.~Yang}, {\em Sparse {G}aussian processes for solving
		nonlinear {PDE}s}, Journal of Computational Physics, 490 (2023), p.~112340.
	
	\bibitem{moreau1965proximite}
	{\sc J.-J. Moreau}, {\em Proximit{\'e} et dualit{\'e} dans un espace
		hilbertien}, Bulletin de la Soci{\'e}t{\'e} math{\'e}matique de France, 93
	(1965), pp.~273--299.
	
	\bibitem{mou2022numerical}
	{\sc C.~Mou, X.~Yang, and C.~Zhou}, {\em Numerical methods for mean field games
		based on {G}aussian processes and {F}ourier features}, Journal of
	Computational Physics,  (2022).
	
	\bibitem{nemirovski2009robust}
	{\sc A.~Nemirovski, A.~Juditsky, G.~Lan, and A.~Shapiro}, {\em Robust
		stochastic approximation approach to stochastic programming}, SIAM Journal on
	optimization, 19 (2009), pp.~1574--1609.
	
	\bibitem{nguyen2019exact}
	{\sc D.~T. Nguyen, M.~Filippone, and P.~Michiardi}, {\em Exact {G}aussian
		process regression with distributed computations}, in Proceedings of the 34th
	ACM/SIGAPP symposium on applied computing, 2019, pp.~1286--1295.
	
	\bibitem{owhadi2019operator}
	{\sc H.~Owhadi and C.~Scovel}, {\em Operator-Adapted Wavelets, Fast Solvers,
		and Numerical Homogenization: From a Game Theoretic Approach to Numerical
		Approximation and Algorithm Design}, vol.~35, Cambridge University Press,
	2019.
	
	\bibitem{quarteroni2008numerical}
	{\sc A.~Quarteroni and A.~Valli}, {\em Numerical approximation of partial
		differential equations}, vol.~23, Springer Science and Business Media, 2008.
	
	\bibitem{quinonero2005unifying}
	{\sc J.~Quinonero-Candela and C.~E. Rasmussen}, {\em A unifying view of sparse
		approximate {G}aussian process regression}, The Journal of Machine Learning
	Research, 6 (2005), pp.~1939--1959.
	
	\bibitem{rahimi2007random}
	{\sc A.~Rahimi and B.~Recht}, {\em Random features for large-scale kernel
		machines.}, in NIPS, vol.~3, Citeseer, 2007, p.~5.
	
	\bibitem{raissi2018numerical}
	{\sc M.~Raissi, P.~Perdikaris, and G.~E. Karniadakis}, {\em Numerical
		{G}aussian processes for time-dependent and nonlinear partial differential
		equations}, SIAM Journal on Scientific Computing, 40 (2018), pp.~A172--A198.
	
	\bibitem{raissi2019physics}
	{\sc M.~Raissi, P.~Perdikaris, and G.~E. Karniadakis}, {\em Physics-informed
		neural networks: A deep learning framework for solving forward and inverse
		problems involving nonlinear partial differential equations}, Journal of
	Computational physics, 378 (2019), pp.~686--707.
	
	\bibitem{rockafellar1970convex}
	{\sc R.~Rockafellar}, {\em Convex analysis}, vol.~18, Princeton university
	press, 1970.
	
	\bibitem{schafer2021sparse}
	{\sc F.~Sch{\"a}fer, M.~Katzfuss, and H.~Owhadi}, {\em Sparse {C}holesky
		factorization by {K}ullback{--}{L}eibler minimization}, SIAM Journal on
	scientific computing, 43 (2021), pp.~A2019--A2046.
	
	\bibitem{shalev2007pegasos}
	{\sc S.~Shai, Y.~Singer, and N.~Srebro}, {\em Pegasos: Primal estimated
		sub-gradient solver for {SVM}}, in Proceedings of the 24th international
	conference on Machine learning, 2007, pp.~807--814.
	
	\bibitem{smale2005shannon}
	{\sc S.~Smale and D.~Zhou}, {\em Shannon sampling {II}: Connections to learning
		theory}, Applied and Computational Harmonic Analysis, 19 (2005),
	pp.~285--302.
	
	\bibitem{thomas2013numerical}
	{\sc J.~W. Thomas}, {\em Numerical partial differential equations: finite
		difference methods}, vol.~22, Springer Science and Business Media, 2013.
	
	\bibitem{tran2020stochastic}
	{\sc Q.~Tran-Dinh, N.~Pham, and L.~Nguyen}, {\em Stochastic gauss-newton
		algorithms for nonconvex compositional optimization}, in International
	Conference on Machine Learning, PMLR, 2020, pp.~9572--9582.
	
	\bibitem{wang2021bayesian}
	{\sc J.~Wang, J.~Cockayne, O.~Chkrebtii, T.~J. Sullivan, and C.~J. Oates}, {\em
		Bayesian numerical methods for nonlinear partial differential equations},
	Statistics and Computing, 31 (2021), pp.~1--20.
	
	\bibitem{wang2019exact}
	{\sc K.~Wang, G.~Pleiss, J.~Gardner, S.~Tyree, K.~Q. Weinberger, and A.~G.
		Wilson}, {\em Exact {G}aussian processes on a million data points}, Advances
	in neural information processing systems, 32 (2019).
	
	\bibitem{wilson2015kernel}
	{\sc A.~Wilson and H.~Nickisch}, {\em Kernel interpolation for scalable
		structured {G}aussian processes ({KISS}-{GP})}, in International conference
	on machine learning, PMLR, 2015, pp.~1775--1784.
	
	\bibitem{yu2016orthogonal}
	{\sc F.~X. Yu, A.~T. Suresh, K.~M. Choromanski, D.~N. Holtmann-Rice, and
		S.~Kumar}, {\em Orthogonal random features}, Advances in neural information
	processing systems, 29 (2016), pp.~1975--1983.
	
	\bibitem{zinkevich2003online}
	{\sc M.~Zinkevich}, {\em Online convex programming and generalized
		infinitesimal gradient ascent}, in Proceedings of the 20th international
	conference on machine learning ({ICML}-03), 2003, pp.~928--936.
	
\end{thebibliography}
\end{document}